\documentclass[11pt]{article}
\usepackage[english]{babel}
\usepackage{latexsym}
\usepackage{epsfig}
\usepackage{graphicx}
\usepackage{comment}
\usepackage{amsfonts,amsmath}
\usepackage{color}
\newtheorem{thm}{Theorem}

\newtheorem{cor}[thm]{Corollary}
\newtheorem{defin}{Definition}
\newtheorem{lem}[thm]{Lemma}

\newcounter{obs}
\newtheorem{obser}[obs]{Observation}
\newtheorem{exa}{Example}
\newenvironment{example}{\begin{exa} \rm }{\hfill$\Box$ \end{exa}}
\newenvironment{definition}{\begin{defin} \rm }{\end{defin}}
\newenvironment{proof}{\begin{trivlist}\item[] \mbox{\it Proof. }}
{\hfill$\Box$ \end{trivlist}}
\def\dv{\mathbb}
\def\ci{\perp\!\!\!\perp}
\def\carr{\mbox{\rm carr}\,}
\def\supp{\mbox{\rm supp}\,}
\def\inte{\mbox{\rm relint}\,}
\def\upA{A\subseteq}
\def\upi{i\subseteq}

\def\upemp{\emptyset\subseteq}
\def\caP{{\cal P}(N)}  
\def\caPM{{\cal P}(M)}
\def\caPp{{\cal P}(N^{\prime})}
\def\caPL{{\cal P}(L)}
\def\caPR{{\cal P}(R)}
\def\caE{{\cal E}(N)}  
\def\caEM{{\cal E}(M)}
\def\caER{{\cal E}(R)}
\def\calK{\lozenge}  
\def\caK{\lozenge(N)}
\def\caKS{\lozenge_{\cal S}(N)}
\def\caKl{\lozenge_{\ell}(N)}
\def\caKu{\lozenge_{u}(N)}
\def\caKo{\lozenge_{o}(N)}
\def\caKw{\lozenge^{*}(N)}
\def\caKRw{\lozenge^{*}(R)}
\def\caKLw{\lozenge^{*}(L)}
\def\caKMl{\lozenge_{\ell}(M)}
\def\caKRl{\lozenge_{\ell}(R)}
\def\caKLl{\lozenge_{\ell}(L)}
\def\caKM{\lozenge(M)}
\def\caL{{\cal L}(N)}
\def\caLM{{\cal L}(M)}
\def\caS{{\cal S}(N)}
\def\calS{{\cal S}}
\def\caF{{\cal F}(N)}
\def\calF{{\cal F}}
\def\caFS{{\cal F}_{\cal S}(N)}
\def\caSl{{\cal S}_{\ell}(N)}
\def\caSu{{\cal S}_{u}(N)}
\def\caSo{{\cal S}_{o}(N)}
\def\calD{{\cal D}}
\def\calI{{\cal I}}
\def\caU{{\cal U}(N)}
\def\mS{m_{\cal S}}
\def\bolpi{\mbox{\boldmath $\pi$}}
\def\smbolpi{\mbox{\footnotesize\boldmath $\pi$}}
\def\boliota{\mbox{\boldmath $\iota$}}
\def\bolsigma{\mbox{\boldmath $\sigma$}_{-1}}
\def\compo{\circ} 
\def\liftMN{T_{M\nearrow N}} 
\def\minoDE{T_{-D|E}} 
\def\maxminM{T_{\max\searrow M}}
\def\coars{T_{\sigma_{-1}}\!}
\def\contrS{T_{S\to s}}
\def\lowmod{T_{\ell\mbox{\rm\scriptsize -mod}}}
\def\uppmod{T_{u\mbox{\rm\scriptsize -mod}}}
\def\lowrepl{T_{\downarrow\downarrow}}
\def\upprepl{T_{\uparrow\uparrow}}
\def\res{r}
\def\ext{m}
\def\fac{\mbox{\sf F}}

\begin{document}
\addtolength{\baselineskip}{1mm}
\parskip 0mm

\title{Basic facts concerning extreme supermodular functions\thanks{This research has been
supported by the grant GA\v{C}R n.\ 16-12010S}}
\author{Milan Studen\'{y}\\[0.5ex]
the Institute of Information Theory and Automation of the CAS}
\date{\today} 
\maketitle

\begin{abstract}
Elementary facts and observations on the cone of supermodular set functions are recalled.
The manuscript deals with such operations with set functions which preserve supermodularity
and the emphasis is put on those such operations which even preserve extremality
(of a supermodular function). These involve a few {\em self-transformations\/} of the cone
of supermodular set functions. Moreover, {\em projections\/} to the (less-dimensional) linear space of
set functions for a {\em subset\/} of the variable set are discussed. Finally, several {\em extensions\/} to the
(more-dimensional) linear space of set functions for a {\em superset\/} of the variable set are shown
to be both preserving supermodularity and extremality.
\end{abstract}
\mbox{permutational} transformation, reflection, lifting, support,
minor, coarsening, contraction,\\
 modular extension, replication

\section{Introduction}
The source of motivation for this technical report is the problem of characterization of extreme
supermodular functions. These functions play very important role in describing
(probabilistic) conditional independence structures; see the open problem from \cite[Theme~6]{stu05}.
Nevertheless, supermodular functions and their mirror images, submodular functions, also appear in various branches
of discrete mathematics and have numerous applications in computer science and optimization
\cite{sha72,top98,ZCJ09,bac10}.
There are criteria to recognize whether a given supermodular function is extreme; one of them
was proposed already in the 1970's \cite{roswei74} and a new elegant alternative has recently
been proposed in \cite{SK16}.

In connection with the application of  extreme supermodular functions in testing conditional independence
implication by means of a computer a catalogue of all (types of) extreme supermodular functions over 5 variables has been
created \cite{stuboukoc00}. The paper \cite{KSTT12} contains several interesting results on (linear) operations
with supermodular functions which preserve extremality. The aim of this technical report is to gather
the results on operations with set functions preserving supermodularity and put the observations from \cite{KSTT12}
in this context. Specifically, the operations from \cite[\S\,4]{KSTT12} are interpreted as special cases of
more general supermodularity-preserving operations. This hopefully will make it possible, in future, to interpret
these operations in geometric terms in connection with game-theoretical concept of the {\em core polytope} of
a supermodular game \cite{sha72,SK16}.

The next four sections of the report contain elementary concepts. Section \ref{sec.self-transformations} is then devoted
to {\em self-transformations}, which are mappings from ${\dv R}^{\caP}$ to itself, where $N$ is a {\em variable set\/}
and $\caP$ its power set. Fundamental such operations are the {\em permutational transformation} and the
{\em reflection transformation}.
A special {\em lifting transformation\/} and a related concept of the {\em support\/} of a function
in ${\dv R}^{\caP}$ are discussed in Section \ref{sec.lifting}. This transformation is a basic
{\em extension operation}, that is, a mapping from ${\dv R}^{\caPM}$ to ${\dv R}^{\caP}$, where $M\subseteq N$.
On the other hand, Section~\ref{sec.projections} is devoted to various {\em projections}, that is,
mappings from ${\dv R}^{\caP}$ to ${\dv R}^{\caPM}$, where $M\subseteq N$.
The important ones are the {\em minor\/}  projections and the {\em coarsening transformation}.
These operations preserve supermodularity but not extremality.
The last two sections are devoted to the other extension operations
besides the lifting transformation. At least four different such operations are mentioned:
two kinds of so-called {\em modular extensions} are introduced and two kinds of so-called {\em replications}
are discussed. All of them are shown to preserve both supermodularity and extremality.

\section{Notation}

Let $N$ be a finite non-empty set of {\em variables}; $n := |N|\geq 2$ and $\caP := \{ A : A\subseteq N\}$.\\
Given $a\in N$, the symbol $a$ will also be used to denote the singleton $\{ a\}$.\\
We consider the space ${\dv R}^{\caP}$ of the dimension $2^{n}$ equipped with scalar product
$$
\langle m,u\rangle := \sum_{A\subseteq N} m(A)\cdot u(A)\qquad
\mbox{for ~$m,u\in {\dv R}^{\caP}$}.
$$
Given $A\subseteq N$, we introduce $\delta_{A}\in {\dv R}^{\caP}$ and $m^{\upA}\in {\dv R}^{\caP}$ by the formulas:
$$
\delta_{A}(S) =\left\{
\begin{array}{cl}
+1 & \mbox{if $S=A$,}\\
0 & \mbox{otherwise},
\end{array}
\right.
~~
m^{\upA}(S) =\left\{
\begin{array}{cl}
+1 & \mbox{if $A\subseteq S$,}\\
0 & \mbox{otherwise},
\end{array}
\right.
~~
\quad \mbox{for $S\subseteq N$.}
$$
The symbol $\caE$ will denote the class of {\em elementary} triplets $\langle a,b|C\rangle$, where
$a,b\in N$ are distinct and $C\subseteq N\setminus \{ a,b\}$. Every such triplet is assigned a vector
$u_{\langle a,b|C\rangle}\in {\dv Z}^{\caP}$, called an {\em elementary imset}, given by
$$
u_{\langle a,b|C\rangle}(S) =\left\{
\begin{array}{cl}
+1 & \mbox{if $S=C$ or $S= a\cup b\cup C$},\\
-1 & \mbox{if $S=a\cup C$ or $S=b\cup C$},\\
0 & \mbox{otherwise}.
\end{array}
\right.
$$
Sometimes, even more general notation will be used: given an ordered triplet $A,B,C$ of pairwise disjoint subsets of $N$, we put
$$
u_{\langle A,B|C\rangle}(S) := \delta_{A\cup B\cup C} +\delta_{C} -\delta_{A\cup C} -\delta_{B\cup C}\,,\quad
\mbox{where canceling terms is possible.}
$$
The symbol $\delta(\star\star)$, where $\star\star$ is a predicate (= statement),
will occasionally denote a zero-one function whose value is $+1$ if the statement $\star\star$ is valid and $0$ otherwise.

\section{Basic definitions}

\begin{defin}\rm
A function $m:\caP \mapsto {\dv R}$, that is, an element of ${\dv R}^{\caP}$, is {\em supermodular\/} if
$$
\forall\, A,B\subseteq N\quad
m(A)+m(B)\leq m(A\cup B)+ m(A\cap B)\,.
$$
A {\em modular} function is defined by equality, the {\em submodular} one by converse inequality.
The symbol $\caK$ will denote the class of supermodular functions over $N$, $\caL$ the class 
of modular functions over $N$.
\end{defin}

Supermodular functions are characterized by means of $n\cdot (n-1)\cdot 2^{n-3}$ inequalities that correspond
to elementary imsets:
$$
\forall\, m\in {\dv R}^{\caP}\qquad
m\in\caK ~\Leftrightarrow ~
[\,\, \forall\,\langle a,b|C\rangle\in\caE \quad \langle m,u_{\langle a,b|C\rangle} \rangle\geq 0 \,].
$$
In particular, $\caK$ is a (rational) polyhedral cone in ${\dv R}^{\caP}$.
In fact, the above inequalities are just the facet-defining inequalities for the cone $\caK$ \cite[Corollary\,11]{KVV10};
they imply $\langle m,u_{\langle A,B|C\rangle} \rangle\geq 0$ for pairwise disjoint sets $A,B,C\subseteq N$.
Analogously, the class of submodular
functions $-\caK$ is a polyhedral cone and $\caL =\caK\cap(-\caK)$ is a linear subspace of ${\dv R}^{\caP}$.

\begin{defin}\rm\label{def.CI-induced}
Every $m\in\caK$ is assigned a  table of scalar products, which is a (non-negative) function on $\caE$:
$$
S_{m}: \langle a,b|C\rangle\in \caE ~\mapsto~ \langle m,u_{\langle a,b|C\rangle}\rangle\,.
$$
The {\em induced independency model} is $\calI(m) := \{\langle a,b|C\rangle\in\caE \,:\,
\langle m,u_{\langle a,b|C\rangle} \rangle =0\,\}$; its complement
$\calD(m) := \{\langle a,b|C\rangle\in\caE \,:\,
\langle m,u_{\langle a,b|C\rangle} \rangle >0\,\}$ is the respective {\em dependency model}.
For $\langle a,b|C\rangle\in\caE$, we will write $a\ci b\,|\,C\,\,[m]$ to denote
$\langle m,u_{\langle a,b|C\rangle} \rangle =0$, that is, $\langle a,b|C\rangle\in\calI(m)$,
meaning $a$ is {\em conditionally independent\/} of $b$ given $C$ with respect to $m$.
The notation $A\ci B\,|\,C\,\,[m]$ for pairwise disjoint $A,B,C\subseteq N$ will have analogous meaning.

Functions $m^{1},m^{2}\in\caK$ are {\em quantitatively equivalent}, or strongly equivalent, if they
induce the same table of scalar products:
$$
m^{1}\approx m^{2} ~:=~ [\,\,\forall\,\langle a,b|C\rangle\in\caE \quad \langle m^{1},u_{\langle a,b|C\rangle}\rangle = \langle m^{2},u_{\langle a,b|C\rangle}\rangle\,].
$$
Functions $m^{1},m^{2}\in\caK$ are {\em qualitatively equivalent}, or model equivalent, if they
induce the same (in)dependency model:
$$
m^{1}\sim m^{2} ~:=~ [\,\,\forall\,\langle a,b|C\rangle\in\caE \quad
\langle m^{1},u_{\langle a,b|C\rangle}\rangle =0 ~\Leftrightarrow~ \langle m^{2},u_{\langle a,b|C\rangle}\rangle =0\,].
$$
\end{defin}
Of course, $m^{1}\approx m^{2} ~\Rightarrow~ m^{1}\sim m^{2}$ for any $m^{1},m^{2}\in\caK$.

\section{Quantitative equivalence and standardizations}\label{sec.standardizations}
Quantitative equivalence is easily characterized in terms of modular functions:
$$
\forall\, m^{1},m^{2}\in\caK\qquad
m^{1}\approx m^{2} ~\Leftrightarrow~ m^{1}-m^{2}\in\caL\,.
$$
Since $\caL$ has the dimension $n+1$ and one of its bases consists of the functions $m^{\upemp}$ and
$m^{\upi}$ for $i\in N$, a more specific characterization is as follows: given $m^{1},m^{2}\in\caK$,
$$
m^{1}\approx m^{2} ~\Leftrightarrow~
\left[\,\,\exists\, k\in {\dv R},\, \rho (i)\in {\dv R}~ \mbox{for $i\in N$} \qquad
 m^{1}-m^{2}=k\cdot m^{\upemp}+\sum_{i\in N} \rho (i)\cdot m^{\upi}\,\,\right]\,.
$$

A consistent way of the choice of a representative within (any) equivalence class of $\approx$ corresponds to fixing a complementary space $\caS$ to $\caL$ in ${\dv R}^{\caP}$, which is a linear subspace of ${\dv R}^{\caP}$ such that $\caS\cap\caL=\{0\}$ and $\caS +\caL ={\dv R}^{\caP}$. In this case
every $m\in\caK$ has a uniquely determined decomposition
\begin{eqnarray*}
m=\mS +l &\mbox{where}& \mS\in  \caKS :=\caK\cap\caS ~~\mbox{and}~~ l\in\caL\,,\\
&&\mbox{which situation will be denoted by $\caK =\caKS\oplus\caL$.}
\end{eqnarray*}
Then $\mS$ is the unique element of $\{r\in\caK : r\approx m\}$
in $\caKS$. In particular,
$$
\forall\, m^{1},m^{2}\in\caK \qquad m^{1}\approx m^{2} ~\Leftrightarrow~ \mS^{1}=\mS^{2}\,.
$$
Three ways of the representative choice, that is, of a {\em standardization}, seem to be suitable:
\begin{description}
\item[the {$\ell$-standardization},] named the {\em lower standardization}, is determined by
\begin{eqnarray*}
\caSl &:=& \{\, m\in {\dv R}^{\caP} :~ m(\emptyset)=0 ~~\&~~ \forall\,i\in N\quad m(i)=0\,\} ~~\mbox{and}~\\
\caKl &:=& \caK\cap\caSl\,.
\end{eqnarray*}
Given $m\in\caK$, the formula for the respective representative $m_{\ell}\in\caKl$ is
\begin{eqnarray}
\lefteqn{\hspace*{-1.3cm}m_{\ell}(T)=m(T)+ k_{\ell}+ \sum_{i\in T}\, \rho_{\ell}(i)~~\mbox{for $T\subseteq N$},}\label{eq.l-stan}\\
&\mbox{where}& k_{\ell} =-m(\emptyset) ~~\mbox{and}~~ \rho_{\ell}(i)= m(\emptyset)-m(i)\quad \mbox{for $i\in N$}\,.\nonumber
\end{eqnarray}
A simple basic observation is that every $\ell$-standardized supermodular function is
\mbox{non-decreasing}, that is, $S\subseteq T ~\Rightarrow~
m_{\ell}(S)\leq m_{\ell}(T)$, and, hence, non-negative.
\item[the {$u$-standardization},] named the {\em upper standardization}, is determined by
\begin{eqnarray*}
\caSu &:=& \{\, m\in {\dv R}^{\caP} :~ m(N)=0  ~~\&~~  \forall\,i\in N\quad m(N\setminus i)=0\,\},\\
\caKu &:=& \caK\cap\caSu\,.
\end{eqnarray*}
Given $m\in\caK$, the formula for the respective representative $m_{u}\in\caKu$ is
\begin{eqnarray}
\lefteqn{\hspace*{-1.8cm}m_{u}(T)=m(T)+ k_{u}+ \sum_{i\in T}\, \rho_{u}(i)~~\mbox{for $T\subseteq N$},}\label{eq.u-stan}\\
&\mbox{where}& k_{u} =(n-1)\cdot m(N)-\sum_{j\in N} m(N\setminus j) \nonumber\\
 &\mbox{and}& \rho_{u}(i)= m(N\setminus i)-m(N)\quad \mbox{for $i\in N$}\,.\nonumber
\end{eqnarray}
Again, one can easily observe that every $u$-standardized supermodular function is non-increasing, that is, $S\subseteq T ~\Rightarrow~
m_{u}(S)\geq m_{u}(T)$, and, hence, non-negative.
\item[the {$o$-standardization},] named the {\em orthogonal standardization}, is determined by
\begin{eqnarray*}
\caSo &:=& \{\, m\in {\dv R}^{\caP} :~ \sum_{S\subseteq N} m(S)=0  ~~\&~~  \forall\,i\in N\quad \sum_{S\subseteq N:\, i\in S} m(S)=0\,\},\\
\caKo &:=& \caK\cap\caSo\,.
\end{eqnarray*}
Note that the requirements are, in fact, $\langle m,m^{\upemp}\rangle=0$ and
$\langle m,m^{\upi}\rangle=0$ for $i\in N$, which means that $\caSo$ is the orthogonal complement to $\caL$.
The formula for the respective representative $m_{o}\in\caKo$, on basis of any $m\in\caK$, is
\begin{eqnarray}
\lefteqn{\hspace*{-0.5cm}m_{o}(T)=m(T)+ k_{o}+ \sum_{i\in T}\, \rho_{o}(i)~~\mbox{for $T\subseteq N$},}\label{eq.o-stan}\\
&\mbox{where}& k_{o} = \frac{1}{2^{n-1}}\cdot \sum_{S\subseteq N} |S|\cdot m(S) - \frac{n+1}{2^{n}}\cdot\sum_{S\subseteq N} m(S)\nonumber\\
&\mbox{and}& \rho_{o}(i)= \frac{1}{2^{n-1}}\cdot \left[\,\sum_{S\subseteq N} m(S)-2\cdot \sum_{S\subseteq N:\, i\in S} m(S) \,\right]\quad \mbox{for $i\in N$}\,.\nonumber
\end{eqnarray}
A notable fact is that in case of an integral (= integer-valued) $\ell$-standardized representative the respective $u$-standardized representative is also integral, but the respective \mbox{$o$-standardized} representative may be fractional. Moreover,
the $o$-standardized representative need not be non-negative.
\end{description}

\noindent {\em Remark}~ There are other options. In the context of
submodular functions and matroid theory the following special {\em polymatroidal\/}
standardization appears to be utilized: put
\begin{eqnarray*}
\calS_{p}(N) &:=& \{\, m\in {\dv R}^{\caP} :~ m(\emptyset)=0 ~~\&~~  \forall\,i\in N\quad m(N\setminus i)=m (N)\,\} ~~\mbox{and}~\\
\calK_{p}(N) &:=& \caK\cap\calS_{p}(N)\,,
\end{eqnarray*}
which, however, in the supermodular context, leads to non-increasing and non-positive representatives of equivalence
classes of of quantitative equivalence $\approx$.

\section{Qualitative equivalence and extremality}\label{sec.qual-extreme}

In this section, we assume the reader is familiar with some basic facts from polyhedral geometry, gathered
in Appendix, \S\,\ref{sec.face}.

\begin{defin}\rm
The symbol $\caF$ will denote the lattice of non-empty faces of\/\, $\caK$ ordered by inclusion $\subseteq$.
We will also call it the {\em face lattice\/} of the cone $\caK$.
For any $m\in\caK$, $F(m)$ will denote the smallest face containing $m$:
$$
F(m) := \bigcap\, \{ F :\ F\in\caF ~\& ~ m\in F\}\,.
$$
\end{defin}

\begin{obser}\rm\label{obs.face}
~~$\forall\,  m^{1},m^{2}\in\caK\qquad
F(m^{1})\subseteq F(m^{2}) ~\Leftrightarrow~ \calI (m^{1})\supseteq\calI (m^{2})$\,.
\end{obser}

\begin{proof}
Let $\calF_{*}(N)$ denote the class of facets of\/ $\caK$. As every face is the intersection of facets containing it, $\forall\, m\in\caK$ one has $F(m)=\bigcap_{F\in\caF :\, m\in F} F=\bigcap_{F\in\calF_{*}(N) :\, m\in F} F$.
Hence, $F(m^{1})\subseteq F(m^{2})$ iff
\begin{equation}
\forall\, F\in\calF_{*}(N)\qquad m^{2}\in F ~\Rightarrow~  m^{1}\in F\,;
\label{eq.face}
\end{equation}
the necessity of \eqref{eq.face} being derived by a contradiction.
However, facets of\/ $\caK$ correspond to elementary imsets; specifically, they have the form
$F=\{ m\in\caK :~ \langle m,u_{\langle a,b|C\rangle}\rangle =0\,\}$ for triplets $\langle a,b|C\rangle\in\caE$.
Therefore, \eqref{eq.face} is equivalent to the condition $\forall\,\langle a,b|C\rangle\in\caE$,
$\langle m^{2},u_{\langle a,b|C\rangle}\rangle =0 ~\Rightarrow~ \langle m^{1},u_{\langle a,b|C\rangle}\rangle =0$, which is exactly\,
$\calI (m^{1})\supseteq\calI (m^{2})$.
\end{proof}

\begin{cor}[characterization of qualitative equivalence]\rm\label{cor.face} ~~\\
~~$\forall\,  m^{1},m^{2}\in\caK\qquad m^{1}\sim m^{2}  ~\Leftrightarrow~
\calI (m^{1})=\calI (m^{2}) ~\Leftrightarrow~ F(m^{1})=F(m^{2})$.\\
Thus, the equivalence classes of $\sim$ are relative interiors of faces of $\caK$. In other words,
$$
\forall\,  m^{1},m^{2}\in\caK\qquad m^{1}\sim m^{2}  ~\Leftrightarrow~
\left[\, \exists\, F\in\caF \quad m^{1},m^{2}\in\inte (F)\,\right]\,.
$$
\end{cor}

\begin{proof}
Indeed, for $F\in\caF$, $\inte (F)$ is the set of $m\in\caK$ such that
$F(m)=F$.
\end{proof}

\medskip
\begin{defin}\rm\label{def.extreme}
We say that a supermodular function $m\in\caK$ is {\em extreme\/} if $F(m)$ is the atom of the
lattice $(\caF,\subseteq)$.
\end{defin}

In other words, $m$ is extreme if it belongs to the relative interior of an atomic face of $\caF$, that is, of a face of $\caF$ of the dimension $\dim(\caL)+1=n+2$.

\begin{obser}\rm\label{obs.isom-stan} ~~\\
Let $\caS$ be a complementary space to $\caL$ in ${\dv R}^{\caP}$ and
let $\caFS$ denote the lattice of non-empty faces of the cone $\caKS :=\caK\cap\caS$. The lattice $(\caFS ,\subseteq )$ is isomorhic to the lattice $(\caF ,\subseteq )$. The correspondence is as follows:
$$
\mbox{every}~ F\in\caF \quad \mbox{has the form $F=F_{\calS}+\caL$ where $F_{\calS}\in\caFS $}.
$$
The inverse relation is $F_{\calS}=F\cap\caS$.
\end{obser}

\begin{proof}
Note that both $\caK$ and $\caKS$ is defined by means of inequalities
of the type $0\leq\langle m,u_{\langle a,b|C\rangle}\rangle$
with $\langle a,b|C\rangle\in\caE$ for $m$ in the respective linear space, which is either
${\dv R}^{\caP}$ or $\caS$.
Thus, any facet of $\caK$, respectively of $\caKS$, is determined by such an inequality.
Because \mbox{every} face is the intersection of facets, every face of $\caK$, respectively of $\caKS$, is determined by an inequality of the form
$0\leq\langle m, v\rangle$, where $v$ belongs to the conic hull
$\{ u_{\langle a,b|C\rangle} :~ \langle a,b|C\rangle\in\caE\,\}$, denoted by\, $\caU$. Note that\, $\caU$\/ is a subset of\/ ${\caL}^{\perp}$, the orthogonal complement of\, $\caL$.

Hence, for any $F\in\caF$, there exists $v\in\caU$ with $F=\{ m\in\caK :~
\langle m,v\rangle =0\}$. Then $\caL\subseteq F$, and, because $\caK =\caKS\oplus\caL$, one has $F=F_{\calS}+\caL$ where
$F_{\calS} :=F\cap\caS$ has the form of a face of\/ $\caKS$, namely
$F_{\calS}=\{ m\in\caKS :~ \langle m,v\rangle =0\}$.

Conversely, $R\in\caFS$ is of the form $R=\{ m\in\caKS :~ \langle m,v\rangle =0\}$ for $v\in\caU$. Then $\caL\subseteq\{ m\in\caK :~ \langle m,v\rangle =0\}$ implies
$R+\caL\subseteq \{ m\in\caK :~ \langle m,v\rangle =0\}$. Since the converse inclusion follows from the decomposition ${\dv R}^{\caP}=\caS\oplus\caL$, one has
$R+\caL =\{ m\in\caK :~ \langle m,v\rangle =0\}\equiv F\in\caF$. The fact that the correspondence respects the inclusion relation $\subseteq$ is trivial.
\end{proof}

\begin{cor}\rm\label{cor.isom-stan} ~~\\
Let $\caS$ be an arbitrary complementary space to $\caL$ in ${\dv R}^{\caP}$. Then $F\subseteq {\dv R}^{\caP}$ is an atom of the lattice $(\caF ,\subseteq )$ iff it has the form $F=R+\caL$, where $R$ is an extreme ray of the (pointed) cone $\caKS =\caK\cap\caS$.
\end{cor}

\begin{proof}
~ $F$ is an atom of $\caF$ iff it has the form $F=R+\caL$, where $R$ is an atom
of $\caFS$. Since $\caKS$ is pointed, the atoms of $\caKS$ are the extreme rays of $\caKS$.
\end{proof}

In other words, $m\in\caK$ is extreme (in sense of Definition \ref{def.extreme}) iff its standardized representative $m_{\calS}$ belongs to the relative interior of an extreme ray of\/ $\caKS$. Specifically, the corresponding equivalence class of\, $\sim$ has the form
$$
\{ r\in\caK :~ r\sim m\} =\{ \alpha\cdot m_{\calS}+l :~ \alpha >0 ~\&~ l\in\caL\,\}\,.
$$
The characterization of extreme supermodular functions is, therefore, equivalent to the characterization of extreme rays of\/ $\caKS$.
In particular, the extremality of a supermodular function can be expressed in terms of any standardization,
this concept does not depend on the choice of a standardization.
\medskip

\noindent {\em Remark}\, The above observation implies that, for extreme supermodular functions,
quantitative and qualitative equivalences are very close. In fact, the following can be considered
as a cryptic equivalent definition of extremality: $m\in\caK$ is extreme iff $m\not\in\caL$ and,
for any $r\in\caK$, one has $m\sim r ~\Leftrightarrow~ [\,\exists\,\alpha>0~~ \alpha\cdot m\approx r\,]$.
Indeed, in case of non-extreme non-modular $m\in\caK\setminus\caL$ write
$\mS =\sum_{i} \alpha_{i}\cdot m^{i}$ where $\alpha_{i}>0$ and $m^{i}$ are at least two generators of
extreme rays in $\caKS$. Then choose one of such $m^{j}$, put $r:=\mS+m^{j}$ and, because
of $m^{j}\in F(m)$, observe that $m\sim r$ while there is no $\alpha>0$ with $\alpha\cdot\mS=r$.
\medskip

\noindent {\em Remark}\, Further important observation is that a necessary condition for the extremality
(of a supermodular function) is the existence of an {\em integral\/} (= integer-valued) representative
in the respective class (of qualitative equivalence). To observe that consider the lower standardization:
the cone $\caKl$ is pointed and every $0\neq m\in\caKl$ satisfies $m(N)>0$. The (bounded) polyhedron
$\{\, m\in\caKl\, :\ m(K)\leq 1\,\}$ is specified by inequalities with rational coefficients. Therefore,
it is a polytope with vertices whose all components are rational numbers. Its non-zero vertices are generators
of extreme rays of $\caKl$ and each of them can be multiplied by a positive integer to obtain a vector
whose all components are integers, that is, an element of ${\dv Z}^{\caP}$.

\section{Self-transformations preserving supermodularity}\label{sec.self-transformations}

In the subsequent text we deal with transformations which preserve supermodularity,
mainly the linear ones, and discuss whether they preserve the extremality (of a supermodular function).
We also ask which of them can be interpreted as transformations of  equivalence classes of\,
$\approx$. This particular section is devoted to transformations of ${\dv R}^{\caP}$ into itself.

\subsection{Permutational transformations}\label{ssec.permutation}

\begin{definition}\rm\label{def.permutation}
Every permutation $\pi : N\to N$ (= a bijective mapping of $N$ onto itself)
can be extended to a bijective mapping $\bolpi$ of $\caP$ onto itself by the relation
$$
\bolpi (S) := \{\, \pi (i)\,:\ i\in S\}\quad
\mbox{for any $S\subseteq N$.}
$$
This step allows one to define a {\em permutational transformation\/} $T_{\pi}: {\dv R}^{\caP}\to
{\dv R}^{\caP}$:
$$
T_{\pi}: ~~ m\in {\dv R}^{\caP}\mapsto\, m\compo\bolpi\in {\dv R}^{\caP},\quad
\mbox{where\quad $m\compo\bolpi\,(S):= m(\bolpi (S))$~~ for $S\subseteq N$.}
$$
\end{definition}

Clearly, the permutational transformation is invertible, its inverse is given
by the inverse permutation: $(T_{\pi})^{-1}= T_{\pi^{-1}}$.
As $T_{\pi}$ maps $\caL$ to $\caL$, $m^{1}\approx m^{1} ~\Rightarrow~ m^{1}\compo\bolpi \approx m^{2}\compo\bolpi$
and $T_{\pi}$ can be viewed as a mapping between equivalence classes of\, $\approx$.

\begin{obser}\rm\label{obs.permut} ~~\\
Let $\pi : N\to N$ be a permutation and $m\in {\dv R}^{\caP}$.
Then
\begin{itemize}
\item[(i)] One has $m\in\caK$ iff $T_{\pi}(m)\equiv m\compo\bolpi\in\caK$ and the independency model
induced by $m\compo\bolpi$ is determined as follows:
\begin{equation}
a\ci b\,|\,C\,\,[m\compo\bolpi] ~~\Leftrightarrow~~
\pi (a)\ci \pi (b)\,|\,\bolpi (C)\,\,[m]
\qquad \mbox{for any $\langle a,b|C\rangle\in\caE$.}
\label{eq.CI-permut}
\end{equation}
In particular, $T_{\pi}$ maps $\caK$ onto $\caK$.
\item[(ii)] The transformation $T_{\pi}$ assigns faces of $\caK$ to faces of $\caK$ while
the inclusion is preserved: $G\subseteq F$ iff\/ $T_{\pi}(G)\subseteq T_{\pi}(F)$
for any $G,F\subseteq {\dv R}^{\caP}$.\\[0.2ex]
In particular, $m\in\caK$ is extreme in $\caK$ iff $T_{\pi}(m)\equiv m\compo\bolpi$ is extreme in $\caK$.
\end{itemize}
\end{obser}

\begin{proof}
(i): since $T_{\pi}$ is invertible and its inversion is a mapping of the same type, namely $T_{\pi^{-1}}$, it is enough to show
$m\in\caK ~\Rightarrow~ m\compo\bolpi\in\caK$. Given $\langle a,b|C\rangle\in\caE$, one has
\begin{eqnarray}
\lefteqn{\langle m\compo\bolpi, u_{\langle a,b|C\rangle}\rangle =
m\compo\bolpi\,(a\cup b\cup C)+ m\compo\bolpi\,(C) -m\compo\bolpi\,(a\cup C) -m\compo\bolpi\,(b\cup C)} \nonumber\\
&=& m(\pi (a)\cup\pi (b)\cup\bolpi(C)) +m(\bolpi(C))
-m(\pi (a)\cup\bolpi(C)) -m(\pi (b)\cup\bolpi(C)) \nonumber\\
&=& \langle m, u_{\langle \pi (a),\pi (b)|\smbolpi (C)\rangle}\rangle\,,
\label{eq.permu-equal}
\end{eqnarray}
and the latter expression is non-negative because
$\langle \pi (a),\pi (b)|\bolpi (C)\rangle\in\caE$ and $m\in\caK$.
The formula \eqref{eq.permu-equal} also implies
\eqref{eq.CI-permut} because $i\ci j\,|\,K\,\,[\tilde{m}]$ means
$\langle \tilde{m}, u_{\langle i,j|K\rangle}\rangle=0$.\\[0.4ex]
(ii): the first step is to show that facets of $\caK$ are transformed by $T_{\pi}$ to facets
of $\caK$. The facets have the form $\{ m\in\caK \,:\ \langle m, u_{\langle i,j|K\rangle}\rangle =0\,\}$ for $\langle i,j|K\rangle\in\caE$. Realize that \eqref{eq.permu-equal} can be re-written
in the form $\langle m, u_{\langle i,j|K\rangle}\rangle=
\langle m\compo\bolpi, u_{\langle \pi^{-1}(i),\pi^{-1}(j)|\smbolpi^{-1}(K)\rangle}\rangle$ and obtain
\begin{eqnarray*}
\lefteqn{\hspace*{-0.9cm}T_{\pi}(\{ m\in\caK :\ \langle m, u_{\langle i,j|K\rangle}\rangle =0\}) =
}\\
&=&\{ m\compo\bolpi\in\caK :\ \langle m\compo\bolpi, u_{\langle \pi^{-1}(i),\pi^{-1}(j)|\smbolpi^{-1}(K)\rangle}\rangle  =0\}\\
&=& \{ r\in\caK :\ \langle r, u_{\langle \pi^{-1}(i),\pi^{-1}(j)|\smbolpi^{-1}(K)\rangle}\rangle  =0\}\,.
\end{eqnarray*}
Since $\langle i,j|K\rangle~\mapsto~\langle \pi^{-1}(i),\pi^{-1}(j)|\smbolpi^{-1}(K)\rangle$
is bijection of $\caE$ onto $\caE$, $T_{\pi}$ transforms facets of $\caK$ to themselves. Because every face of $\caK$ is the intersection of facets containing it and
$T_{\pi}(\bigcap_{\gamma} F_{\gamma})=\bigcap_{\gamma} T_{\pi}(S_{\gamma})$, faces of
$\caK$ come to faces of $\caK$. Since $T_{\pi}$ is bijective, the inclusion (of faces) is preserved.
Thus, the correspondence of faces of $\caK$ determined by $T_{\pi}$ is an auto-isomorphism of the face lattice\/ $\caF$.
The second claim in (ii) is an easy consequence: atoms of $\caF$ must be transformed to atoms of $\caF$.
\end{proof}

\noindent {\em Remark}\, Permutational transformations respect three main standardizations from \S\,\ref{sec.standardizations}:
$$
\forall\, m\in\caK\qquad
m_{\ell}\compo\bolpi =(m\compo\bolpi)_{\ell}, \quad
m_{u}\compo\bolpi =(m\compo\bolpi)_{u}, \quad
m_{o}\compo\bolpi =(m\compo\bolpi)_{o} \,.
$$
Indeed, it is easy to see that any complementary space $\caS$ among $\caSl$, $\caSu$ and $\caSo$ satisfies the following condition:
$$
\forall\, r\in {\dv R}^{\caP}\qquad r\in\caS ~~\Rightarrow ~~ r\compo\bolpi\in\caS\,.
$$
To show this implies $m_{\calS}\compo\bolpi=(m\compo\bolpi)_{\calS}$ for any
$m\in\caK$ consider a decomposition $m=m_{\calS}+l$ with $l\in\caL$,
using linearity of $T_{\pi}$ write $m\compo\bolpi=m_{\calS}\compo\bolpi+l\compo\bolpi$ and realize
$l\compo\bolpi\in\caL$. The uniqueness of the decomposition of $m\compo\bolpi$ implies
what is desired.

\subsection{Reflection transformation}\label{ssec.reflection}

\begin{definition}\rm\label{def.reflection}
We consider a reflection mapping $\boliota :\caP\to\caP$
defined by
$$
\boliota (S) := N\setminus S\qquad \mbox{for~ $S\subseteq N$,}
$$
which is a bijective mapping of\/ $\caP$ onto itself.\\[0.2ex]
This allows one to define the {\em reflection transformation\/}
$T_{\iota}: {\dv R}^{\caP}\to {\dv R}^{\caP}$:
\begin{eqnarray*}
\lefteqn{\hspace*{-1cm}T_{\iota}: ~~ m\in {\dv R}^{\caP}\mapsto\, m\compo\boliota\in {\dv R}^{\caP},}\\
&&\mbox{where\quad $m\compo\boliota\, (S):= m(\boliota (S))\equiv m(N\setminus S)$~~ for $S\subseteq N$.}
\end{eqnarray*}
Given $m\in {\dv R}^{\caP}$, its {\em reflection} is the function $T_{\iota}(m)\equiv m\compo\boliota$.
\end{definition}

Evidently, the reflection transformation is invertible, its inverse is itself: $(T_{\iota})^{-1}= T_{\iota}$.
Since $T_{\iota}$ maps $\caL$ to $\caL$, $m^{1}\approx m^{1}$ implies $m^{1}\compo\boliota \approx m^{2}\compo\boliota$
and the reflection can be interpreted as a transformation of equivalence classes of\, $\approx$.

\begin{obser}\rm\label{obs.reflect} ~~ 
Assume $m\in {\dv R}^{\caP}$. Then
\begin{itemize}
\item[(i)] One has $m\in\caK$ iff $T_{\iota}(m)\equiv m\compo\boliota\in\caK$
and the independency model induced by $m\compo\boliota$ is determined as follows:
\begin{equation}
a\ci b\,|\,C\,\,[m\compo\boliota] ~~\Leftrightarrow~~
a\ci b\,\,|\,\,N\setminus (a\cup b\cup C)\,\,[m]
\qquad \mbox{for any $\langle a,b|C\rangle\in\caE$.}
\label{eq.CI-reflect}
\end{equation}
In particular, $T_{\iota}$ maps $\caK$ onto $\caK$.
\item[(ii)] The transformation $T_{\iota}$ assigns faces of $\caK$ to
faces of $\caK$ while the inclusion is preserved: $G\subseteq F$ iff\/ $T_{\iota}(G)\subseteq T_{\iota}(F)$
for any $G,F\subseteq {\dv R}^{\caP}$.\\[0.2ex]
In particular, $m\in\caK$ is extreme in $\caK$ iff $T_{\iota}(m)\equiv m\compo\boliota$ is extreme in $\caK$.
\end{itemize}
\end{obser}

Note that the last fact from Observation \ref{obs.reflect}(ii) was noticed
in \cite[\S\,5.1.3]{stuboukoc00}, \cite[Lemma\,4.1]{KSTT12}.

\begin{proof}
(i): since $T_{\iota}$ is self-invertible, it is enough to show
$m\in\caK ~\Rightarrow~ m\compo\boliota\in\caK$. Given
$\langle a,b|C\rangle\in\caE$ and $D:= N\setminus (a\cup b\cup C)$, one has
\begin{eqnarray}
\langle m\compo\boliota, u_{\langle a,b|C\rangle}\rangle &=&
m\compo\boliota\, (a\cup b\cup C)+ m\compo\boliota\, (C) -m\compo\boliota\, (a\cup C)
-m\compo\boliota\, (b\cup C) \nonumber\\
&=& m(D) +m(a\cup b\cup D) -m(b\cup D) -m(a\cup D) \nonumber\\
&=& \langle m, u_{\langle a,b|D\rangle}\rangle =
\langle m, u_{\langle a,b|N\setminus (a\cup b\cup C)\rangle}\rangle\,,
\label{eq.reflect-equal}
\end{eqnarray}
and the latter expression is non-negative because
$\langle a,b|D\rangle\in\caE$ and $m\in\caK$.
The formula \eqref{eq.reflect-equal} also implies
\eqref{eq.CI-reflect} because $i\ci j\,|\,K\,\,[\tilde{m}]$ means
$\langle \tilde{m}, u_{\langle i,j|K\rangle}\rangle=0$.\\[0.4ex]
(ii): we first show that facets of $\caK$, having the form
$\{ m\in\caK \,:\ \langle m, u_{\langle i,j|K\rangle}\rangle =0\,\}$
for $\langle i,j|K\rangle\in\caE$, are transformed by $T_{\iota}$ to facets of $\caK$.
Realize that \eqref{eq.reflect-equal} also says $\langle m, u_{\langle i,j|K\rangle}\rangle =
\langle m\compo\boliota, u_{\langle i,j|N\setminus (i\cup j\cup K)}\rangle$ and obtain
\begin{eqnarray*}
\lefteqn{\hspace*{-3.4cm}T_{\iota}(\{ m\in\caK :\ \langle m, u_{\langle i,j|K\rangle}\rangle =0\}) =
\{ m\compo\boliota\in\caK :\ \langle m\compo\boliota, u_{\langle i,j|N\setminus (i\cup j\cup K)\rangle}\rangle  =0\}}\\
&=& \{ r\in\caK :\ \langle r, u_{\langle i,j|N\setminus (i\cup j\cup K)\rangle}\rangle  =0\}\,.
\end{eqnarray*}
Since $\langle i,j|K\rangle~\mapsto~\langle i,j\,|\,N\setminus (i\cup j\cup K)\rangle$
is bijection of $\caE$ onto $\caE$, $T_{\iota}$ transforms facets of $\caK$ to themselves.
Because every face of $\caK$ is the intersection of facets containing it and
$T_{\iota}$ is bijective, faces of $\caK$ come to faces of $\caK$ while their inclusion
is preserved. Hence, the correspondence of faces of $\caK$ given by $T_{\iota}$ is an auto-isomorphism
of the face lattice\/ $\caF$ and the second claim in (ii) is evident.
\end{proof}

\noindent {\em Remark}\, The reflection transformation behaves to the standardizations from \S\,\ref{sec.standardizations} as follows:
$$
\forall\, m\in\caK\qquad
m_{\ell}\compo\boliota =(m\compo\boliota)_{u}, \quad
m_{u}\compo\boliota =(m\compo\boliota)_{\ell}, \quad
m_{o}\compo\boliota =(m\compo\boliota)_{o} \,.
$$
Indeed, it is easy to see that
$$
\forall\, r\in {\dv R}^{\caP}\qquad r\in\caSl ~\Leftrightarrow ~ r\compo\boliota\in\caSu, \quad r\in\caSo ~\Leftrightarrow ~ r\compo\boliota\in\caSo \,.
$$
To show the former equivalence implies $m_{\ell}\compo\boliota =(m\compo\boliota)_{u}$ for any $m\in\caK$ write
$m=m_{\ell}+l$ with $l\in\caL$ and by linearity of $T_{\iota}$ get
$m\compo\boliota =m_{\ell}\compo\boliota +l\compo\boliota$.
Since $m_{\ell}\compo\boliota\in\caSu$ and $l\compo\boliota\in\caL$,
the uniqueness of the decomposition of $m\compo\boliota$ with respect to $\caSu$
implies what is needed. The proof of the other formulas is analogous.

\subsection{Monotonizations}

There are transformations of $\caK$ into itself which are not injective
but ascribe monotone (supermodular) functions to supermodular functions.

\subsubsection{Lower and upper standardizations as linear monotonizations}
In fact, the $\ell$-standardization from \S\,\ref{sec.standardizations}
can be viewed as a linear transformation from $\caK$ to $\caK$
which ascribes a non-decreasing function $m_{\ell}$ to any $m\in\caK$.
Analogously, the $u$-standardization is a linear transformation from $\caK$ to $\caK$ which ascribes a non-increasing function $m_{u}$ to $m\in\caK$.

It follows from arguments in \S\,\ref{sec.qual-extreme} that these two transformations preserve extremality, that is,
they ascribe extreme supermodular functions to extreme ones. However, since they assign quantitatively equivalent
functions to supermodular functions they are not useful if one is interested in generating new (= different)
types of extreme supermodular functions.

\subsubsection{Non-linear maximum-based monotonizations}
There are also non-linear transformations of\/ $\caK$ into itself which
ascribe monotone functions to functions from $\caK$; they are defined through
{\em maximization\/} over the class of subsets, respectively of supersets.
Specifically, given $m\in\caK$, let us put
\begin{eqnarray*}
m^{\max\subseteq}(S) &:=&
\max\,\{\, m(T)\,:\ T\subseteq S\}\qquad
\mbox{for $S\subseteq N$,~~ and }\\
m^{\max\supseteq}(S) &:=&
\max\,\{\, m(T)\,:\ T\supseteq S\}\qquad
\mbox{for $S\subseteq N$.}
\end{eqnarray*}
Analogues of these transformations in the context of submodular functions have been pinpointed
in \cite[Prop.\,19,\,20]{bac10}, where mimimization over subsets/supersets is applied instead.

\begin{obser}\rm\label{obs.monotone} ~~\\
If $m\in\caK$ then
$m^{\max\subseteq}$ is non-decreasing function in $\caK$ and
$m^{\max\supseteq}$ is non-increasing function in $\caK$ .
\end{obser}

\begin{proof}
Given $A,B\subseteq N$, there exist $A^{\prime}\subseteq A$ and
$B^{\prime}\subseteq B$ such that $m^{\max\subseteq}(A)=m(A^{\prime})$ and
$m^{\max\subseteq}(B)=m(B^{\prime})$. Thus, we write using supermodularity of $m$:
\begin{eqnarray*}
m^{\max\subseteq}(A)+ m^{\max\subseteq}(B) &=& m(A^{\prime})+m(B^{\prime})
~\leq~ m (A^{\prime}\cup B^{\prime}) + m (A^{\prime}\cap B^{\prime})\\
&\leq& \max\,\{\, m(T) : T\subseteq A\cup B\,\} +
\max\,\{\, m(T) : T\subseteq A\cap B\,\}\\
&=&
m^{\max\subseteq}(A\cup B) + m^{\max\subseteq}(A\cap B)\,,
\end{eqnarray*}
because $A^{\prime}\cup B^{\prime}\subseteq A\cup B$ and
$A^{\prime}\cap B^{\prime}\subseteq A\cap B$.
The proof for $m^{\max\supseteq}$ is analogous.
\end{proof}

The above-defined maximum-based monotonizations can be viewed as
mutually dual; specifically, they are reflections of each other in the sense
$$
(m\compo\boliota)^{\max\subseteq}= m^{\max\supseteq}\compo\boliota
~~\mbox{and}~~ (m\compo\boliota)^{\max\supseteq}= m^{\max\subseteq}\compo\boliota
\qquad \mbox{for every $m\in\caK$.}
$$
However, they do not preserve extremality of a supermodular
function, although they may generate distinct types of extreme functions.

\begin{example}\label{exa.monot}\rm
Consider $N=\{ a,b,c\}$ and put
$$
m_{0} ~:=~ 2\cdot\delta_{\{ a,b,c\}} +\delta_{\{ a,b\}} +\delta_{\{ a,c\}} +\delta_{\{ b,c\}},
$$
which an extreme non-decreasing function in $\caKl$.
Quantitatively equivalent function $m_{1}:=m_{0}-\frac{1}{2}\cdot m^{a\subseteq}
-\frac{1}{2}\cdot m^{b\subseteq}\approx m_{0}$ is not non-decreasing.
The monotonization of $m_{1}$ (the maximization over subsets) gives a non-extreme non-decreasing supermodular function
$$
m_{1}^{\max\subseteq} = \delta_{\{ a,b,c\}} +\frac{1}{2}\cdot\delta_{\{ a,c\}}
+\frac{1}{2}\cdot\delta_{\{ b,c\}} =
\frac{1}{2}\cdot(\delta_{\{ a,b,c\}} +\delta_{\{ a,c\}}) +\frac{1}{2}\cdot(\delta_{\{ a,b,c\}} +\delta_{\{ b,c\}}).
$$
On the other hand, putting $m_{2} := m_{0}-m^{c\subseteq}$ results in
$m_{2}^{\max\subseteq} =\delta_{\{ a,b,c\}} + \delta_{\{ a,b\}}$, which is an extreme function in $\caKl$,
but of another permutation type than the original $m_{2}\approx m_{0}$.
To complete the picture note that putting $m_{3} := m_{0}-\frac{2}{3}\cdot m^{a\subseteq} -\frac{2}{3}\cdot m^{b\subseteq} -\frac{2}{3}\cdot m^{c\subseteq}$ gives
$m_{3}^{\max\subseteq}\equiv 0$, which is a modular function.
\end{example}
\smallskip

The above example shows that $m\approx r \not\Rightarrow m^{\max\subseteq}\approx  r^{\max\subseteq}$;
thus, the monotonization cannot be interpreted as a mapping between equivalence classes of\, $\approx$.
\medskip

\noindent {\em Remark}\,
The observation that the maximum-based monotonization can ``generate"
a different extreme type from a given extreme type (of supermodular functions)
leads to the following idea. Let us introduce a binary relation $\rightsquigarrow$ between
$\sim$-equivalence classes:
$$
\mbox{for $m,r\in\caK$,}\qquad m\rightsquigarrow r ~~\mbox{if}\quad
[\,\exists\, m_{*}\approx m ~~\mbox{such that}~~ m_{*}^{\max\subseteq}\sim r\,].
$$
The question is whether this relation, respectively its transitive closure, offers a sensible complexity
comparison between distinct extreme types of supermodular functions. For example, in case
$N=\{ a,b,c\}$ one has $m_{0}\rightsquigarrow \delta_{N}$ for $m_{0}$ from Example
\ref{exa.monot} while $\neg [\delta_{N}\rightsquigarrow m_{0}]$; the same holds
for $m^{\{a,b\}\subseteq}$ in place of $\delta_{N}$.


\subsection{Outer composition and multiplication}
An example of a transformation preserving supermodularity in the $\ell$-standardized
frame is the outer composition with certain convex functions. An example of a binary operation
on $\caKl$ is pointwise multiplication. These facts follow from the next observation,
mentioned in more general context in \cite[Lemma\,\,2.6.4]{top98}.

\begin{lem}\rm \label{lem.outer-compo}
Let $m,r\in\caKl$ and $g:[0,\infty)\times[0,\infty)\to {\dv R}$ be a real function of two
variables, whose one-variable sections are non-decreasing and convex on $[0,\infty )$ and which satisfies
\begin{equation}
\forall\, x\leq x^{\prime},\ y\leq y^{\prime}\qquad
 g(x^{\prime},y) + g(x,y^{\prime})\leq g(x^{\prime},y^{\prime}) + g(x,y)\,.
\label{eq.outer-compo}
\end{equation}
Then the composed function $S\subseteq N\mapsto g(m(S),r(S))$ is supermodular on $\caP$.\\
In particular, if $h:[0,\infty)\rightarrow {\dv R}$ is a non-decreasing convex function with $h(\emptyset )=0$
then $m\in\caKl ~\Rightarrow~ (h\compo m)\in\caKl$.
\end{lem}

\begin{proof}
Consider $A,B\subseteq N$ and the goal is to show
$$
0\leq g(m(A\cup B),r(A\cup B))  -g(m(A),r(A))
- g(m(B),r(B)) +g(m(A\cap B),r(A\cap B)).
$$
Let us denote $\bar{m}:= m(A\cup B)+m(A\cap B)-m(A)$ and $\bar{r}:= r(A\cup B)+r(A\cap B)-r(A)$;
since $m$ and $r$ are both supermodular and non-decreasing one has $m(B)\leq \bar{m}\leq m(A\cup B)$ and
$r(B)\leq \bar{r}\leq r(A\cup B)$.
The goal is reached by summing the following six inequalities:
\begin{eqnarray*}
\mbox{(a)}~~~0&\leq& g(m(A\cup B),r(A\cup B))  -g(m(A\cup B),\bar{r})
- g(m(A),r(A\cup B)) +g(m(A),\bar{r}),\\
\mbox{(b)}~~~0&\leq& g(m(A\cup B),\bar{r})  -g(m(A\cup B),r(A\cap B))
-g(\bar{m},\bar{r}) +g(\bar{m},r(A\cap B)),\\
\mbox{(c)}~~~0&\leq& g(m(A),r(A\cup B))  -g(m(A),\bar{r})
- g(m(A),r(A)) +g(m(A),r(A\cap B)),\\
\mbox{(d)}~~~0&\leq& g(m(A\cup B),r(A\cap B)) -g(\bar{m},r(A\cap B))\\
&& \hspace*{1cm}-\,g(m(A),r(A\cap B)) +g(m(A\cap B),r(A\cap B)),\\
\mbox{(e)}~~~0&\leq& g(\bar{m},\bar{r})  -g(\bar{m},r(B)),\\
\mbox{(f)}~~~0&\leq& g(\bar{m},r(B)) -g(m(B),r(B)).
\end{eqnarray*}
Indeed, most of the terms cancel after summing;
specifically, $\mp g(m(A\cup B),\bar{r})$ in (a) and (b),
$\mp g(m(A),r(A\cup B)) \pm g(m(A),\bar{r})$ in (a) and (c),
$\mp g(m(A\cup B),r(A\cap B)) \pm g(\bar{m},r(A\cap B))$ in (b) and (d),
$\mp g(\bar{m},\bar{r})$ in (b) and (e),
$\pm g(m(A),r(A\cap B))$ in (c) and (d), and
$\mp g(\bar{m},r(B))$ in (e) and (f). It remains to observe the validity
of those inequalities (a)-(f).

The inequality (a) is \eqref{eq.outer-compo}  with $x=m(A)$, $x^{\prime}=m(A\cup B)$,
$y=\bar{r}$ and $y^{\prime}=r(A\cup B)$; note that $x\leq x^{\prime}$ and
$y\leq y^{\prime}$ because $m$ is non-decreasing with respect to inclusion.
The inequality (b) is \eqref{eq.outer-compo} with $x=\bar{m}$, $x^{\prime}=m(A\cup B)$,
$y=r(A\cap B)$ and $y^{\prime}=\bar{r}$; again $x\leq x^{\prime}$ and
$y\leq y^{\prime}$ for $r$ is non-decreasing.
The inequality (c) follows from the convexity of the function
$y\in [0,\infty )\mapsto \hat{g}(y):= g(m(A),y)$. Specifically, the convexity
of $\hat{g}$ means that it has non-decreasing increments, which can
be formulated in a symmetric way as follows:
$$
\forall\, y\in [0,\infty),~\Delta\geq 0, ~\delta\geq 0,\qquad
0\leq \hat{g}(y+\Delta +\delta)- \hat{g}(y+\Delta)-\hat{g}(y+\delta)+\hat{g}(y)\,.
$$
Thus, putting $y=r(A\cap B)$, $\Delta =r(A\cup B)-r(A)$ and $\delta = r(A)-r(A\cap B)$
results in (c). The inequality (d) follows from
the convexity of the function $x\in [0,\infty )\mapsto \tilde{g}(x):= g(x,r(A\cap B))$;
where, analogously, $x=m(A\cap B)$, $\Delta =m(A\cup B)-m(A)$, $\delta =m(A)-m(A\cap B)$.
Finally, (e) follows from the assumption that the function $y\in [0,\infty)\mapsto g(\bar{m},y)$
is non-decreasing and (f) from the fact that $x\in [0,\infty)\mapsto g(x,r(B))$ is non-decreasing.\\[0.3ex]
The second statement (about the function $h$ on $[0,\infty )$) follows easily from the
main one by putting $g(x,y):=h(x)$ for $x,y\geq 0$ and $r:=m$.
\end{proof}

\begin{cor}\rm\label{cor.multiplic}
Given $m^{1},m^{2}\in\caKl$, their pointwise multiple
$$
m^{3}(S) := m^{1}(S)\cdot m^{2}(S)\qquad \mbox{for $S\subseteq N$,}
$$
is also in $\caKl$. Analogously, if $m^{1},m^{2}\in\caKu$ then $m^{3}\in\caKu$.
\end{cor}

\begin{proof}
Use Lemma \ref{lem.outer-compo}, where $m=m^{1}$, $r=m^{2}$ and
$g(x,y)= x\cdot y$ for $x,y\geq 0$. The statement about $u$-standardized functions
then follows from Observation \ref{obs.reflect}(i) and the reflection correspondence between $\caKu$ and $\caKl$:
$m^{i}\in\caKu ~\Rightarrow~ m^{i}\compo\boliota\in\caKl$, $i=1,2$, and
one has $m^{3}\compo\boliota =(m^{1}\compo\boliota)\cdot (m^{2}\compo\boliota)\in\caKl
~\Rightarrow~ m^{3}\in\caKu$.
\end{proof}
\smallskip

However, neither the multiplication nor the outer composition with a convex function
preserves extremality of a supermodular function. Consider $N=\{ a,b,c\}$ and $m_{0}$ from
Example \ref{exa.monot}. The multiplication of $m_{0}$ with itself results in
$$
\hat{m}_{0} = 4\cdot\delta_{\{ a,b,c\}} +\delta_{\{ a,b\}} +\delta_{\{ a,c\}} +\delta_{\{ b,c\}} =
2\cdot\delta_{\{ a,b,c\}} + (2\cdot\delta_{\{ a,b,c\}} +\delta_{\{ a,b\}} +\delta_{\{ a,c\}} +\delta_{\{ b,c\}}),
$$
which is not extreme in $\caK$. Of course, $\hat{m}_{0}$ is the outer composition of $m_{0}$
with $h(x):= x^{2}$.

\section{Lifting and support}\label{sec.lifting}

In this section we deal with a simple linear mapping to a higher-dimensional space
which preserves supermodularity and even extremality of a supermodular function.
Then we introduce a related characteristic of equivalence classes of supermodular
functions.

\subsection{Lifting transformation}\label{ssec.lift-tran}

\begin{defin}\rm\label{def.lifting}
Given $M\subseteq N$, the {\em lifting transformation\/} (from $M$ to $N$) is
a mapping $\liftMN : {\dv R}^{\caPM}\to {\dv R}^{\caP}$ defined as follows:
\begin{eqnarray}
\lefteqn{\hspace*{-7mm}\liftMN :\quad \res\in {\dv R}^{\caPM}~\mapsto ~\liftMN(\res)\equiv \ext~\in {\dv R}^{\caP},}\nonumber \\
&& \mbox{where~~ $[\liftMN (\res )](S)=\ext(S) := \res(S\cap M)$ ~~ for $S\subseteq N$.}
\label{eq.lift-def}
\end{eqnarray}
Given $\res\in {\dv R}^{\caPM}$ the function $\ext\equiv\liftMN (\res)$ is called the {\em lifting\/} of $\res$
(to $N$). A function $\ext\in {\dv R}^{\caP}$ will be called a {\em lifting from $M$} if
$\res\in {\dv R}^{\caPM}$ exists with $\liftMN (\res ) =\ext$.
\end{defin}

Clearly, if $\ext$ is a lifting of $\res\in {\dv R}^{\caPM}$ then $\res$ is the restriction of $\ext$ to $\caPM$.
In particular, $\liftMN$ is an injective mapping from ${\dv R}^{\caPM}$ to ${\dv R}^{\caP}$.
Because $\liftMN$ maps linearly $\caLM$ to $\caL$, $r^{1}\approx r^{2}$ implies
$\liftMN (r^{1})\approx \liftMN (r^{2})$ and the lifting can be interpreted as a transformation
of equivalence classes of\, $\approx$.

\begin{obser}\rm\label{obs.lift-basic} ~~
Assume $M\subseteq N$ and $\res\in {\dv R}^{\caPM}$.
\begin{itemize}
\item[(i)] Then $\res\in\caKM$ iff $\liftMN (\res)\in\caK$
and the independency model induced by $\liftMN (\res)$ is determined as follows: for any $\langle a,b|C\rangle\in\caE$,
\begin{equation}
a\ci b\,|\,C\,\,[\,\liftMN (\res)\,] ~~\Leftrightarrow~~ \{\, a\not\in M ~~\mbox{or}~~ b\not\in M ~~\mbox{or}~~
a\ci b\,|\,C\cap M\,\,[\res]\,\,\}\,.
\label{eq.lift-basic}
\end{equation}
In particular, $\liftMN$ maps $\caKM$ into $\caK$.
\item[(ii)] The lifting transformation commutes with basic standardizations:
given $\res\in\caKM$,\\
$(\liftMN (\res))_{\ell}=\liftMN (\res_{\ell})$, $(\liftMN (\res))_{u}=\liftMN (\res_{u})$,
$(\liftMN (\res))_{o}=\liftMN (\res_{o})$.
\end{itemize}
\end{obser}

\begin{proof}
(i): given $\langle a,b|C\rangle\in\caE$, $\res\in {\dv R}^{\caPM}$ and $\ext=\liftMN(\res)$, one has
\begin{eqnarray}
\lefteqn{\hspace*{-1cm}\langle\, \liftMN (\res), u_{\langle a,b|C\rangle}\rangle =\langle \ext, u_{\langle a,b|C\rangle}\rangle =
\ext(a\cup b\cup C)+ \ext(C) -\ext(a\cup C)-\ext(b\cup C)} \nonumber\\
&=& \res((a\cup b\cup C)\cap M) +\res(C\cap M) -\res((a\cup C)\cap M) -\res((b\cup C)\cap M) \nonumber\\
&=&  \delta (\,(a\cup b)\subseteq M\,)\cdot \langle \res, u_{\langle a,b\,|\, C\cap M\rangle}\rangle
\,,
\label{eq.lift-equal}
\end{eqnarray}
and, if $\res\in\caKM$ then the latter expression is non-negative as $\langle a,b\,|\,C\cap M\rangle\in\caEM$
in case $(a\cup b)\subseteq M$; \eqref{eq.lift-equal} also implies
\eqref{eq.lift-basic} because $i\ci j\,|\,K\,\,[\tilde{m}]$ means
$\langle \tilde{m}, u_{\langle i,j|K\rangle}\rangle=0$.\\[0.4ex]
(ii):
a basic observation is this: if a linear transformation $T$ maps $\caLM$ into $\caL$ and
a complementary space $\calS (M)$ into $\caS$ (see \S\,\ref{sec.standardizations}) then, given $\res\in\caKM$,
$\res=\res_{\calS}+l$ with $\res_{\calS}\in\calS (M)$, $l\in\caLM$ one has $T(\res)=T(\res_{\calS})+T(l)$ with
$T(\res_{\calS})\in\caS$, $T(l)\in\caL$
implying that $T(\res)_{\calS}=T(\res_{\calS})$ for any $\res\in\caKM$. We leave the reader to verify that
$\liftMN$ maps $\calS_{\ell}(M)$ to $\caSl$, $\calS_{u}(M)$ to $\caSu$ and $\calS_{o}(M)$ to $\caSo$.
\end{proof}

\noindent {\em Remark}\, However, not every standardization commutes with the lifting transformation.
Consider the following {\em weird standardization}, given by the complementary space
$$
\calS_{w}(N) := \{\, m\in {\dv R}^{\caP} :~ m(\emptyset)=0 ~~\&~~  \forall\,i\in N\quad m(i)=-m (N)\,\}\quad
\mbox{to~ $\caL$.}
$$
Put $N=\{ a,b,c\}$, $M=\{ a,b\}$ and $r=\delta_{\{ a,b\}}\in\caKM$. One has
$r_{w}=\frac{1}{3}\cdot\delta_{\{ a,b\}} -\frac{1}{3}\cdot\delta_{a}-\frac{1}{3}\cdot\delta_{b}$
while $\liftMN (r)=m^{\{ a,b\}\subseteq}=\delta_{\{ a,b,c\}}+\delta_{\{ a,b\}}$ with
$$
m^{\{ a,b\}\subseteq}_{w} ~=~
\frac{1}{4}\cdot\delta_{\{ a,b,c\}} +\frac{1}{2}\cdot\delta_{\{ a,b\}}
-\frac{1}{2}\cdot\delta_{\{ a,c\}} -\frac{1}{2}\cdot\delta_{\{ b,c\}}
-\frac{1}{4}\cdot\delta_{a} -\frac{1}{4}\cdot\delta_{b} -\frac{1}{4}\cdot\delta_{c}\,,
$$
which differs from the lifting of $r_{w}$ to $N$.
\medskip

The lifting transformation preserves extremality of a supermodular function. However,
one cannot use a direct argument to show that because faces are not mapped to faces.
Nevertheless, considering suitable standardization, for example the $\ell$-standardization,
allows one to overcome this technical obstacle.
The following auxiliary observation is used.

\begin{lem}\rm\label{lem.lift}
If $M\subseteq N$ then
\begin{eqnarray*}
\fac (M) &:=& \{\, \ext\in\caK\, :\ \ext_{\ell}(S)=\ext_{\ell} (S\cap M) ~~\mbox{for any $S\subseteq N$}\,\}\\
&=& \{\, \ext\in\caK\, :\ \langle \ext, u_{\langle i,j|C\rangle}\rangle=0 ~~\mbox{for $i\in N\setminus M$, $j\in N\setminus i$, $C\subseteq N\setminus (i\cup j)$}\,\}
\end{eqnarray*}
is a face of $\caK$. The face lattice of $\caKM$ is isomorphic to the face lattice of $\fac(M)$:
a~face $F$ of $\caKM$ is assigned the least face of $\caK$ containing $\liftMN (F)$ of the following form
$$
F\in \calF (M)  ~\longmapsto~ \{\, \liftMN (\res)+\sum_{i\in N\setminus M} \alpha_{i}\cdot m^{\upi}\,\,:\
\res\in F ~\&~ \alpha_{i}\in {\dv R}~~ \mbox{for $i\in N\setminus M$}\,\}\subseteq \fac(M).
$$
\end{lem}

\begin{proof}
Because $\ext\approx \ext_{\ell}$ for any $\ext\in {\dv R}^{\caP}$, to
verify the equality of two expressions for $\fac(M)$ it is enough to show, for $\ext=\ext_{\ell}\in\caKl$, that
\begin{eqnarray*}
\lefteqn{\hspace*{-1cm}[\, \ext(S)=\ext(S\cap M) ~\mbox{for any $S\subseteq N$}\,]}\\
&\Leftrightarrow&
[\,\langle \ext, u_{\langle i,j|C\rangle}\rangle=0 ~\mbox{for $i\in N\setminus M$, $j\in N\setminus i$, $C\subseteq N\setminus (i\cup j)$}\,].
\end{eqnarray*}
The necessity of the latter condition is evident as $\ext(i\cup j\cup C)=\ext((j\cup C)\cap M)=\ext(j\cup C)$ and
$\ext(i\cup C)=\ext(C\cap M)=\ext(C)$ for $i\in N\setminus M$.
To verify the sufficiency, we first show by induction on $|C|$ that
$\ext(i\cup C)=\ext(C)$ for $i\in N\setminus M$ and $C\subseteq M$.
Indeed, the assumption $\ext$ is $\ell$-standardized implies that in case $|C|=0$; if $|C|>0$ choose
$j\in C$ and write
$$
0 =\langle \ext, u_{\langle i,j|C\setminus j\rangle}\rangle =
\ext(i\cup C) -\ext(C) \underbrace{-\ext(i\cup (C\setminus j))+ \ext(C\setminus j)}_{0}= \ext(i\cup C) -\ext(C)
$$
by the induction premise. Second, we show by induction on $|T|$ that $\ext(T\cup C)=\ext(C)$ for
$T\subseteq N\setminus M$ and $C\subseteq M$. Indeed, we know that in case $|T|\leq 1$; if $|T|>1$ then
choose distinct $i,j\in T$, denote $T^{\prime}=T\setminus (i\cup j)$ and write
\begin{eqnarray*}
\lefteqn{0 =\langle \ext, u_{\langle i,j|T^{\prime}\cup C\rangle}\rangle}\\
 &=&
\ext(T\cup C) \underbrace{-\ext(i\cup T^{\prime}\cup C) -\ext(j\cup T^{\prime}\cup C)+ \ext(T^{\prime}\cup C)}_{-2\cdot \ext(C)+\ext(C)}= \ext(T\cup C)-\ext(C)
\end{eqnarray*}
by the induction premise. This concludes the proof of sufficiency. The latter expression for $\fac(M)$ implies that it is a face of $\caK$.

To observe that the face lattices of $\caKM$ and $\fac (M)$ are isomorphic we introduce
$$
\fac_{\ell}(M) := \fac(M)\cap\caSl =
\{ \ext\in \caKl\,: \ \ext_{\ell}(S)=\ext_{\ell} (S\cap M) ~~\mbox{for any $S\subseteq N$}\,\},
$$
which is, by the above alternative description of $\fac (M)$, a face of $\caKl$.
The proof of isomorphy is based on Observation \ref{obs.isom-stan} and follows the following line:
$$
\caKM ~\longleftrightarrow~ \caKMl
~\longleftrightarrow~ \fac_{\ell}(M) ~\longleftrightarrow~  \fac (M)\,.
$$
Firstly, Observation \ref{obs.isom-stan} applied to $N=M$ and $\ell$-standardization says
that $\caKM$ and $\caKMl$ have isomorphic face lattices. Secondly, it is easy to see that
$$
\fac_{\ell}(M) = \{\, \liftMN (\res)\,:\ \res\in\caKMl\,\},
$$
that is, $\liftMN$ maps bijectively $\caKMl$ onto $\fac_{\ell}(M)$. Therefore, $\liftMN$
defines an isomorphism of face latices of $\caKMl$ and $\fac_{\ell}(M)$.
Thirdly, Observation \ref{obs.isom-stan} applied to $N$ and $\caS =\caSl$ says that
the face $\fac(M)$ of $\caK$ corresponds to $\fac_{\ell}(M)$ and
allows one to deduce that the face lattice of $\fac_{\ell}(M)$ and
that of $\fac(M)$ are isomorphic. Putting these three facts together implies
what is desired.

It remains to realize that the face ascribed to $F\in\calF (M)$
has just the above described form. This also follows from Observation \ref{obs.isom-stan}:
\begin{eqnarray*}
\lefteqn{\hspace*{-2.5cm}F\in \calF (M),\, F\subseteq\caKM  ~\mapsto~ F\cap\caKMl ~\mapsto~ \{ \liftMN (\res)\,:\ \res\in F\cap\caKMl\}}\\
&\mapsto& \{ \liftMN (\res)+l\,:\ \res\in F\cap\caKMl ~\&~ l\in\caL\,\} \\
&& ~= ~\{ \liftMN (\res)+l\,:\ \res\in F ~\&~ l\in\caL\,\}\,,
\end{eqnarray*}
which is clearly the least face of $\caK$ containing $\liftMN (F)$.
The fact $\caLM\subseteq F$ leads to a further simplification of the latter expression.
\end{proof}

\begin{cor}\rm\label{cor.lift-exteme}
Assume $M\subseteq N$ and $\ext\in\caK$ be a lifting of $\res\in\caKM$ to $N$.\\
Then $\ext$ is extreme in $\caK$ iff $\res$ is extreme in $\caKM$.
\end{cor}

\begin{proof}
This follows from the second claim in Lemma \ref{lem.lift} because atomic faces of $\fac(M)$ are just the
atomic faces of $\caK$ belonging to $\fac(M)$.
\end{proof}

An analogous observation for $\ell$-standarized functions only was derived in \cite[Lemma~4.3]{KSTT12}.

\subsection{Support}
The lifting transformation allows one to introduce a set characteristic for any
$m\in\caK$. Its definition is not immediate: actually, the ``right" concept
appears to be a characteristic of an equivalence class of quantitative equivalence $\approx$.
That concept even occurs to be a characteristic of an equivalence class of qualitative
equivalence $\sim$.

\begin{obser}\label{obs.carrier}\rm  ~~
Assume $m\in {\dv R}^{\caP}$.
\begin{itemize}
\item[(i)]
There exists the least set $M\subseteq N$ such that
$m$ is a lifting from $M$, that is,
\begin{equation}
m (S) = m (S\cap M)\qquad \mbox{for every $S\subseteq N$}.
\label{eq.carrier}
\end{equation}
Let's call this least set $M\subseteq N$ the {\em carrier\/} of $m$ and denote by $\carr (m)$.
\item[(ii)] Given $m\in\caK$, one has
$$
\carr (m_{\ell})\subseteq \carr (m),\qquad
\carr (m_{u})\subseteq \carr (m),\qquad
\carr (m_{o})\subseteq \carr (m).
$$
In particular, there exists the least carrier within $\{\, r\in\caK\, :\, r\approx m\}$, namely
the set $\carr (m_{\ell})$, which coincides with $\carr (m_{u})$ and with $\carr (m_{o})$.
\end{itemize}
\end{obser}

\begin{proof}
(i): if both $M_{1}\subseteq N$ and $M_{2}\subseteq N$ satisfy \eqref{eq.carrier}, then
their intersection does so:
$$
m (S) = m(S\cap M_{1}) = m((S\cap M_{1})\cap M_{2}) = m(S\cap (M_{1}\cap M_{2}))
\quad \mbox{for every $S\subseteq N$,}
$$
where, first, \eqref{eq.carrier} for $M_{1}$ is applied and then \eqref{eq.carrier} for $M_{2}$ with $S\cap M_{1}$
in place of $S$ is used. Similarly, one can easily observe that if $m$ satisfies \eqref{eq.carrier} with
$M_{1}$ and $M_{1}\subseteq M_{2}\subseteq N$ then $m$ satisfies \eqref{eq.carrier} with $M_{2}$.
In particular, the (non-empty) class of $M\subseteq N$ such that \eqref{eq.carrier} holds for $M$
is closed under intersection and supersets; therefore, it has the least element.\\[0.4ex]
(ii): put $M =\carr (m)$ and realize that $m$ is the lifting from $M$ of its restriction $r$ to $\caPM$.
By Observation \ref{obs.lift-basic}(ii), $m_{\ell}$ is the lifting from $M$ of $r_{\ell}$.
In particular, $\carr (m_{\ell})\subseteq M$ and, because $m_{\ell}\approx m$, $\carr (m_{\ell})$ is the
least carrier within $\{ r\in\caK\, :\ r\approx m_{\ell}\}$. Analogous arguments can be applied in the
cases of $u$-standardization and $o$-standardization.
\end{proof}

The term ``carrier" we use to name our auxiliary concept was taken over from \cite{roswei74}.
As one can expect, the carrier is not an invariant of an equivalence class of $\approx$:
take $N=\{ a,b,c\}$, $m_{1}=m^{\{a,b\}\subseteq}$ and $m_{2}=m_{1}+m^{c\subseteq}$ in which case
$m_{1}\approx m_{2}$ but $\carr (m_{1})=\{a,b\}\neq \{ a,b,c\}=\carr (m_{2})$.
Thus, the substantial concept seems to be the following one.

\begin{defin}[support]\rm ~\\
Given $m\in\caK$,  the least carrier within the equivalence class $\{\, r\in\caK\, :\, r\approx m\}$
will be called the {\em support\/} of $m$ and denoted by $\supp (m)$.
\end{defin}

Thus, Observation \ref{obs.carrier}(ii) claims, in fact, that the support exists
for any $m\in\caK$ and, moreover, $\supp (m) =\carr (m_{\ell})
=\carr (m_{u})=\carr (m_{o})$. A deeper fact is as follows.

\begin{obser}\label{obs.support}\rm  ~~
Given $m^{1},m^{2}\in\caK$, one has
$$
m^{1}\sim m^{2} \quad\Rightarrow\quad \supp (m^{1})=\supp (m^{2})\,.
$$
\end{obser}

\begin{proof}
Given $m^{1},m^{2}\in\caK$ with $m^{1}\sim m^{2}$, Corollary \ref{cor.face} says
there exists a face $F\in\caF$ of $\caK$ such that $m^{1},m^{2}\in\inte (F)$. Of course,
this face $F$ is uniquely determined by either $m^{1}$ or $m^{2}$. Another note is that $\inte (F)$ is
closed under quantitative equivalence, because $m\approx r$ implies $m\sim r$ and one can
apply Corollary \ref{cor.face} to $m$ and $r$.

Take the above face $F\in\caF$ and consider a set function $m\in\inte (F) \mapsto \carr (m)$
on its interior. This function must have some minimizer with respect to set inclusion;
let $M$ be the corresponding set-minimal carrier, that is,
\begin{equation}
M=\carr (m) ~~\mbox{ for some $m\in\inte (F)$} \quad\&\quad
\forall\, r\in\inte (F) \quad \neg[\,\carr (r)\subset M\,]\,.
\label{eq.supp1}
\end{equation}
We are going to show
\begin{equation}
\forall\, r\in F \qquad \carr (r_{\ell})\subseteq M\,.
\label{eq.supp2}
\end{equation}
Indeed, since $\inte (F)$ is closed under $\approx$-equivalence and $m\approx m_{\ell}$, Observation \ref{obs.carrier}(ii) and \eqref{eq.supp1} allows us to deduce $M=\carr (m_{\ell})$. Thus, $m_{\ell}$ is a lifting from $M$ and $m$ belongs to the special face $\fac (M)$ of $\caK$ introduced in Lemma \ref{lem.lift}.
In particular,
$$
m\in\inte (F)\cap\fac (M)\subseteq \underbrace{F\cap\fac (M)}_{F^{\prime}}.
$$
The sub-face $F^{\prime}$ of $F$ cannot be a proper sub-face of $F$, because otherwise, by the definition
of $\inte (F)$ a contradictory conclusion $m\not\in\inte (F)$ is derived. Hence, $F^{\prime}=F$, which
is another way of saying $F\subseteq \fac (M)$. That implies, for every $r\in F$ and $S\subseteq N$, that
$r_{\ell}(S)=r_{\ell}(S\cap M)$. By the definition of the carrier, $\carr (r_{\ell})\subseteq M$ and
\eqref{eq.supp2} is verified.

Now, \eqref{eq.supp2} combined with \eqref{eq.supp1} implies $\carr (r_{\ell})=M$ for any $r\in\inte (F)$.
In particular, the set-function $r\in\inte (F) \mapsto \carr (r_{\ell})\equiv\supp (r)$ is constant
on $\inte (F)$, taking the value $M$. Thus, one has
$\supp (m^{1})=M=\supp (m^{2})$ for original $m^{1},m^{2}\in\caK$.
\end{proof}

\section{Projections}\label{sec.projections}
In this section we discuss supermodularity-preserving transformations
from ${\dv R}^{\caP}$ to a lower-dimensional space ${\dv R}^{\caPM}$ where $|M|\leq |N|$.

\subsection{Linear projections: minors and their mean}\label{ssec.minors}
A natural linear mapping from ${\dv R}^{\caP}$ to ${\dv R}^{\caPM}$, $M\subseteq N$, is the {\em restriction\/}
of $m\in{\dv R}^{\caP}$ to $\caPM$. It can be viewed as a complementary operation to the lifting because
its application to $\liftMN (r)$ returns back $r\in {\dv R}^{\caPM}$.
However, the restriction is not the only such linear mapping.
One can alternatively ascribe to $m\in{\dv R}^{\caP}$ the function $r\in {\dv R}^{\caPM}$ given by
\begin{equation}
r(S) := m(S\cup (N\setminus M))\qquad \mbox{for $S\subseteq M$,}
\label{eq.def-conditioning}
\end{equation}
which defines a different projection of ${\dv R}^{\caP}$ to ${\dv R}^{\caPM}$. Both these projections can be interpreted
as special cases of a more general linear projection.

\begin{defin}\rm\label{def.minor}
Assume $M\subseteq N$ and $N\setminus M=D\cup E$ with $D\cap E=\emptyset$, that is,
the sets $D$, $E$ and $M$ form a partitioning of $N$.
The {\em minor projection with deleting of $D$ and extracting of $E$} is
a mapping $\minoDE : {\dv R}^{\caP}\to {\dv R}^{\caPM}$ defined as follows:
\begin{eqnarray}
\lefteqn{\minoDE :\quad m\in {\dv R}^{\caP}~\mapsto ~\minoDE (m)\equiv r\in {\dv R}^{\caPM},}\nonumber \\
&& \mbox{where~~ $[\minoDE (m)](S)=r(S) := m(S\cup E)$ ~~ for $S\subseteq M$.}
\label{eq.minor-def}
\end{eqnarray}
Given $m\in {\dv R}^{\caP}$ the function $r\equiv\minoDE (m)$ is called the {\em minor\/} of $m$
obtained by deleting of $D$ and extracting of $E$.
A function $r\in {\dv R}^{\caPM}$, $M\subseteq N$, will be called a {\em minor} of $m\in {\dv R}^{\caP}$
if a partitioning of $N\setminus M$ into sets $D$ and $E$ exists such that $\minoDE (m) =r$.
\end{defin}

The terminology is motivated by matroid theory, where these operations
(with rank functions of matroids) are named analogously.
Because $\minoDE$ maps linearly $\caL$ to $\caLM$, $m^{1}\approx m^{2}$ implies
$\minoDE (m^{1})\approx \minoDE (m^{2})$; thus, $\minoDE$ can be interpreted as a transformation
of equivalence classes of\, $\approx$. Minor projections preserve supermodularity.

\begin{obser}\label{obs.minor}\rm  ~~
Assume $M\subseteq N$ and $N\setminus M=D\cup E$ with $D\cap E=\emptyset$.
If $m\in\caK$ then $\minoDE (m)\in\caKM$ and the independency model induced by $ \minoDE (m)$ is as follows:
\begin{equation}
a\ci b\,|\,C\,\,[\,\minoDE (m)\,] ~~\Leftrightarrow~~
a\ci b\,|\,C\cup E\,\,[m]\qquad \mbox{for any $\langle a,b|C\rangle\in\caEM$.}
\label{eq.minor-indep}
\end{equation}
In particular, $\minoDE$ maps linearly $\caK$ into $\caKM$.
\end{obser}

\begin{proof}
Assume $m\in\caK$ and put $r:= \minoDE (m)$. Given $\langle a,b|C\rangle\in\caEM$
\begin{eqnarray}
\lefteqn{\hspace*{-1cm}\langle\, \minoDE (m), u_{\langle a,b|C\rangle}\rangle = \langle r, u_{\langle a,b|C\rangle}\rangle =
r(a\cup b\cup C)+ r(C) -r(a\cup C) -r(b\cup C)} \nonumber\\
&=& m(a\cup b\cup C\cup E)+ m(C\cup E) -m(a\cup C\cup E) -m(b\cup C\cup E) \nonumber\\
&=& \langle m, u_{\langle a,b|C\cup E\rangle}\rangle\geq 0,
\label{eq.minor-equal}
\end{eqnarray}
because $\langle a,b|C\cup E\rangle\in\caE$ and $m\in\caK$.
The formula \eqref{eq.minor-equal} also implies
\eqref{eq.minor-indep} because $i\ci j\,|\,K\,\,[\tilde{m}]$ means
$\langle \tilde{m}, u_{\langle i,j|K\rangle}\rangle=0$.
\end{proof}

The above discussed restriction to $\caPM$ can be interpreted as the
{\em minor projection with deleting of $N\setminus M$}, that is, $D=N\setminus M$, $E=\emptyset$
(no extraction part); a simplified notation can be $T_{-D}$ then.
In the context of the corresponding operation with (induced) independency models,
given by \eqref{eq.minor-indep} with $E=\emptyset$, the name {\em marginalization} is common.

Similarly, the mapping \eqref{eq.def-conditioning} is called
the {\em minor projection with extracting of $N\setminus M$}, that is, $D=\emptyset$, $E=N\setminus M$
(no deletion part); shorter simpler notation could be $T_{|E}$.
Some books on matroid theory use the term ``contraction" instead of ``extraction", but
we found that misleading because ``contraction"  much better fits another operation which discussed later in \S\,\ref{ssec.contraction}.
In the context of the operation with independency models,
given by \eqref{eq.minor-indep} with $E=N\setminus M$, the name {\em conditioning}
has often been used.

It is easy to see that any minor projection is a linear transformation complementary to the lifting
transformation in the sense that $\minoDE\compo\liftMN$ is identity, formally
$$
\forall\, M\subseteq N, D\cup E=N\setminus M, D\cap E=\emptyset\qquad
\minoDE \,(\,\liftMN (r)\,) = r ~~\mbox{for any $r\in {\dv R}^{\caPM}$.}
$$
However, the minor projections do not preserve extremality of a supermodular function.

\begin{example}\label{exa.minors}\rm
Consider $N=\{ a,b,c,d\}$, $M=\{ a,b,c\}$ and put
$$
m_{*} ~:=~ 2\cdot\delta_{\{ a,b,c,d\}} +\delta_{\{ a,b,c\}} +\delta_{\{ a,b,d\}} +\delta_{\{ a,c,d\}},
$$
which an extreme non-decreasing function in $\caKl$. The extraction of $E=\{ d\}$ gives
$$
T_{|d}(m_{*}) = 2\cdot\delta_{\{ a,b,c\}} +\delta_{\{ a,b\}} +\delta_{\{ a,c\}} =
(\delta_{\{ a,b,c\}} +\delta_{\{ a,b\}}) \,+\, (\delta_{\{ a,b,c\}} +\delta_{\{ a,c\}})\,,
$$
which is not an extreme supermodular function. Another example is
$$
r_{*} \,:=\, 3\cdot\delta_{\{ a,b,c,d\}} +2\cdot\delta_{\{ a,b,c\}} +2\cdot\delta_{\{ a,b,d\}} +2\cdot\delta_{\{ a,c,d\}}
+\delta_{\{ b,c,d\}} +\delta_{\{ a,b\}} +\delta_{\{ a,c\}}  +\delta_{\{ a,d\}},
$$
which is also an extreme supermodular function in $\caKl$. The deletion of $D=\{ d\}$
gives the same result $T_{-d}(r_{*})
=(\delta_{\{ a,b,c\}} +\delta_{\{ a,b\}}) \,+\, (\delta_{\{ a,b,c\}} +\delta_{\{ a,c\}})$.
\end{example}

\noindent {\em Remark}\, Another linear projection from ${\dv R}^{\caP}$ to ${\dv R}^{\caPM}$, $M\subseteq N$, is as follows: assign to $m\in {\dv R}^{\caP}$ the function $r\in {\dv R}^{\caPM}$ given by the formula
\begin{equation}
r(S) \,:=\, \frac{2^{|M|}}{2^{|N|}}\cdot
\sum_{L\subseteq N\setminus M} m(S\cup L) \qquad \mbox{for any $S\subseteq M$}.
\label{eq.mean-minor}
\end{equation}
This linear mapping can be interpreted as the {\em mean-minor projection} because
\eqref{eq.mean-minor} can be re-written as
$$
r(S) = \frac{1}{2^{|N\setminus M|}}\cdot \sum_{D,E\subseteq N\setminus M,\, D\cap E=\emptyset}
[\minoDE (m)](S)\qquad \mbox{for $S\subseteq M$}.
$$
Being a convex combination of (all) minor projections, this mean-minor projection also
maps linearly $\caK$ to $\caKM$ and is complementary to the lifting transformation.
The independency model induced by $r$ given by \eqref{eq.mean-minor} is determined as follows:
$$
a\ci b\,|\,C\,\,[\,r\,] ~~\Leftrightarrow~~
\left\{\,\forall\, L\subseteq N\setminus M\quad  a\ci b\,|\,C\cup L\,\,[m]\,\right\}
\qquad \mbox{for any $\langle a,b|C\rangle\in\caEM$.}
$$
As one can expect, the mean-minor projection does not preserve extremality; one
can follow Example \ref{exa.minors} and observe that $m_{*}=2\cdot\delta_{\{ a,b,c,d\}} +\delta_{\{ a,b,c\}} +\delta_{\{ a,b,d\}} +\delta_{\{ a,c,d\}}$ has the projection
$$
\frac{1}{2}\cdot(\delta_{\{ a,b,c\}} +\delta_{\{ a,b\}}) \,+\, \frac{1}{2}\cdot(\delta_{\{ a,b,c\}} +\delta_{\{ a,c\}}) \,+\, \frac{1}{2}\cdot\delta_{\{ a,b,c\}} \quad \mbox{to $M=\{a ,b,c\}$}\,.
$$
\smallskip

An additional comment on the linear projections is that the {\em deletion} $T_{-N\setminus M}$
preserves $\ell$-standardization, the {\em extraction} $T_{|N\setminus M}$ preserves $u$-standardization
and the {\em mean-minor} projection preserves $o$-standardization; but they do not preserve the
other standardizations. Thus, it seems these three linear projections ``correspond" to the basic standardizations.

\subsection{Coarsening and contraction}\label{ssec.contraction}
Besides minor projections, there is another linear supermodularity-preserving mapping
to a lower-dimensional space.
To illustrate the idea consider a set $S\subseteq N$ of variables and imagine
we wish to ``contract" the set $S$ into a single variable $s\in S$. The following concept
of coarsening generalizes the idea of the contraction operation.

\begin{defin}\rm\label{def.contraction}
Let $N,M$ be sets of variables and $\sigma :N\to M$ a mapping onto $M$.
It induces a pre-image mapping $\bolsigma :\caPM\to\caP$.
The {\em coarsening transformation} determined by $\sigma$ is
a mapping $\coars : {\dv R}^{\caP}\to {\dv R}^{\caPM}$ defined as follows:
\begin{equation}
\coars: m\in {\dv R}^{\caP}\mapsto\, m\compo\bolsigma\in {\dv R}^{\caPM},~
\mbox{where}~ m\compo \bolsigma (T) := m(\sigma_{-1}(T)) ~~\mbox{for $T\subseteq M$.}
\label{eq.contract-def}
\end{equation}
A special case of this operation is as follows: given
$s\in S\subseteq N$ and $N^{\prime}:= (N\setminus S)\cup s\subseteq N$
the {\em contraction\/} (of $S$ to $s$) is the mapping
$\contrS : {\dv R}^{\caP}\to {\dv R}^{\caPp}$ defined as the coarsening
transformation determined the mapping $\sigma: N\to N^{\prime}$ given by
$$
\sigma (t) =\left\{
\begin{array}{lll}
s && \mbox{if~ $t\in S$},\\
t && \mbox{if~ $t\in N\setminus S$},
\end{array}
\right.
\qquad \mbox{ for $t\in N$.}
$$
\end{defin}
\smallskip

In particular, given $m\in {\dv R}^{\caP}$, $s\in S\subseteq N$ and $r:= \contrS (m)$, one has
$$
r(L)=m(L) ~~\mbox{and}~~ r(L\cup s)=m(L\cup S)\qquad
\mbox{for $L\subseteq N\setminus S$},
$$
which operation clearly differs from the minor operation, either that with deleting of \mbox{$S\setminus s$} and
that with extracting of $S\setminus s$. Because $\coars$ maps linearly $\caL$ to $\caLM$, $m^{1}\approx m^{2}$ implies
$\coars (m^{1})\approx \coars (m^{2})$ and the coarsening can be interpreted as a transformation
of equivalence classes of\, $\approx$; in particular, the same holds for the contraction.

Another comment is that the permutation transformation
$T_{\pi}$ from \S\,\ref{ssec.permutation} can be viewed as a special case of the coarsening
transformation $\coars$, with $\sigma$ being the inverse of $\pi$.
The coarsening also preserves supermodularity.

\begin{obser}\label{obs.coarsening}\rm  ~
Assume that $\sigma :N\to M$ a mapping onto $M$.
If $m\in\caK$ then $\coars (m)\in\caKM$ and the independency model induced by $\coars (m)$ is
as follows:
\begin{equation}
a\ci b\,|\,C\,\,[\,\coars (m)\,] ~\Leftrightarrow~
\sigma_{-1}(a)\ci \sigma_{-1}(b)\,|\,\sigma_{-1}(C)\,\,[m]\qquad \mbox{for any $\langle a,b|C\rangle\in\caEM$.}
\label{eq.coars-indep}
\end{equation}
In particular, given $s\in S\subseteq N$,
$\contrS$ maps $\caK$ to $\calK(N^{\prime})$ with $N^{\prime}:= (N\setminus S)\cup s\subseteq N$.
\end{obser}

\begin{proof}
Assume $m\in\caK$ and put $r:= \coars (m)$. Given $\langle a,b|C\rangle\in\caEM$
\begin{eqnarray}
\lefteqn{\hspace*{-1cm}\langle\, \coars (m)\,,\, u_{\langle a,b|C\rangle}\rangle = \langle r, u_{\langle a,b|C\rangle}\rangle =
r(a\cup b\cup C)+ r(C) -r(a\cup C) -r(b\cup C)} \nonumber\\
&=& m(\sigma_{-1}(a\cup b\cup C))+ m(\sigma_{-1}(C)) -m(\sigma_{-1}(a\cup C))
-m(\sigma_{-1}(b\cup C)) \nonumber\\
&=& m(\,\sigma_{-1}(a)\cup \sigma_{-1}(b)\cup \sigma_{-1}(C)\,)+ m(\,\sigma_{-1}(C)\,) -m(\,\sigma_{-1}(a)\cup \sigma_{-1}(C)\,) \nonumber\\
&& -\,\,m(\,\sigma_{-1}(b)\cup \sigma_{-1}(C)\,) ~=~ \langle m, u_{\langle \sigma_{-1}(a),\sigma_{-1}(b)\,|\,\sigma_{-1}(C)\rangle}\rangle\geq 0,
\label{eq.coars-equal}
\end{eqnarray}
because $\langle \sigma_{-1}(a),\sigma_{-1}(b)\,|\,\sigma_{-1}(C)\rangle$ is a triplet of
pairwise disjoint subsets of $N$ and $m\in\caK$.
The formula \eqref{eq.coars-equal} also implies
\eqref{eq.coars-indep} because $i\ci j\,|\,K\,\,[\tilde{m}]$ means
$\langle \tilde{m}, u_{\langle i,j|K\rangle}\rangle=0$.
\end{proof}

Thus, the contraction transformation $\contrS$, where $s\in S\subseteq N$,
can be interpreted as a supermodularity-preserving projection
and also viewed as a complementary transformation to $\liftMN$ where $M=(N\setminus S)\cup s$.
Nevertheless, it does not preserve
the extremality of a supermodular function:
follow Example \ref{exa.minors} with
$m_{*}=2\cdot\delta_{\{ a,b,c,d\}} +\delta_{\{ a,b,c\}} +\delta_{\{ a,b,d\}} +\delta_{\{ a,c,d\}}$,
take $S=\{ c,d\}$, $s=c$, and observe that
$\contrS(m_{*})=(\delta_{\{ a,b,c\}} +\delta_{\{ a,c\}}) +\delta_{\{ a,b,c\}}$,
which is not extreme.

\subsection{Non-linear maximum-based projection}
There is a non-linear supermodularity-preserving mapping from ${\dv R}^{\caP}$ to ${\dv R}^{\caPM}$, $M\subseteq N$.

\begin{defin}\rm\label{def.max-minor}
Given $M\subseteq N$, the {\em max-minor projection\/} is a mapping\\
$\maxminM : {\dv R}^{\caP}\to {\dv R}^{\caPM}$ defined as follows:
\begin{eqnarray}
\lefteqn{\maxminM :\quad m\in {\dv R}^{\caP}~\mapsto ~\maxminM (m)\equiv r\in {\dv R}^{\caPM},}\nonumber \\
&& \mbox{where~ $[\maxminM (m)](S)=r(S) := \max\, \{\, m(S\cup L)\, :L\subseteq N\setminus M\}$ ~ for $S\subseteq M$.~~~}
\label{eq.maxminor-def}
\end{eqnarray}
Given $m\in {\dv R}^{\caP}$ the function $r\equiv\maxminM (m)$ 
will be called the {\em max-minor\/} of $m$ for $M$.
\end{defin}

The motivation for the terminology is clear: one can re-write \eqref{eq.maxminor-def} as follows:
$$
[\maxminM (m)](S) \equiv r(S) = \max_{D,E\subseteq N\setminus M,\, D\cap E=\emptyset} ~~
[\minoDE (m)](S)\qquad \mbox{for $S\subseteq M$},
$$
that is, $r$ is a point-wise maximum on $\caPM$ of all minors of $m$.
It also follows from this interpretation that the max-minor projection can be viewed
as complementary to the lifting $\liftMN$. Clearly, $\maxminM$ is not
a linear mapping, but it preserves supermodularity.

\begin{obser}\label{obs.max-minor}\rm  ~~
If $M\subseteq N$ and $m\in\caK$ then $\maxminM (m)\in\caKM$.
\end{obser}

\begin{proof}
Denote $r:=\maxminM (m)$.
Given $A,B\subseteq M$, there exist $A^{\prime},B^{\prime}\subseteq N\setminus M$ such that
$r(A)=m(A\cup A^{\prime})$ and $r(B)=m(B\cup B^{\prime})$. Thus, we write using supermodularity of $m$,
\begin{eqnarray*}
r(A)+ r(B) &=& m(A\cup A^{\prime})+m(B\cup B^{\prime})\\
&\leq& m (\,(A\cup A^{\prime})\cup (B\cup B^{\prime})\,) + m (\,(A\cup A^{\prime})\cap (B\cup B^{\prime})\,)\\
&=& m (\,(A\cup B)\cup (A^{\prime}\cup B^{\prime})\,) + m (\,(A\cap B)\cup (A^{\prime}\cap B^{\prime})\,)\\
&\leq&
\max_{L\subseteq N\setminus M}\, m(\,(A\cup B)\cup L\,)  ~~+~
\max_{L\subseteq N\setminus M}\, m(\,(A\cap B)\cup L\,)\\
&=&
r(A\cup B) + r(A\cap B)\,,
\end{eqnarray*}
because $A^{\prime}\cup B^{\prime}\subseteq N\setminus M$ and
$A^{\prime}\cap B^{\prime}\subseteq N\setminus M$.
\end{proof}

The above fact has been mentioned in more general context in \cite[Theorem\, 2.7.6]{top98}.
A minimum-based analogue of the discussed transformation
has been applied in the context of submodular functions: it has been named ``partial minimum" in \cite{bac10}
and utilized in the definition of the concept of ``expressibility" in \cite{ZCJ09}.
\smallskip

The max-minor projection does not preserve quantitative equivalence and cannot be interpreted
as a transformation between $\approx$-equivalence classes. Follow Example \ref{exa.monot} with $N=\{ a,b,c\}$, $M=\{ a,b\}$ and
$m_{0}=2\cdot\delta_{\{ a,b,c\}} +\delta_{\{ a,b\}} +\delta_{\{ a,c\}} +\delta_{\{ b,c\}}$.
Then $m_{0}\approx m_{0}-m^{c\subseteq}$ despite
$\maxminM (m_{0})=2\cdot\delta_{\{ a,b\}}+\delta_{a}+\delta_{b}\not\approx \delta_{\{ a,b\}}=
\maxminM (m_{0}-m^{c\subseteq})$.
Neither the max-minor projection preserves extremality of a supermodular function.
Follow Example~\ref{exa.minors} with $N=\{ a,b,c,d\}$,
$m_{*}=2\cdot\delta_{\{ a,b,c,d\}} +\delta_{\{ a,b,c\}} +\delta_{\{ a,b,d\}} +\delta_{\{ a,c,d\}}$
and observe that its max-minor projection $2\cdot\delta_{\{ a,b,c\}} +\delta_{\{ a,b\}} +\delta_{\{ a,c\}}$
to $M=\{ a,b,c\}$ is not extreme.

\section{Product-based composition and modular extensions}\label{sec.modul-ext}
The topic of this section is a special product-based binary operation with set functions which
preserves both supermodularity and extremality. Two linear transformations, interpreted as modular extensions
of set functions, can be derived from that binary operation.

\subsection{Product-based composition}

Consider the cone $\caKw$ of non-negative non-decreasing supermodular functions, which is a pointed polyhedral cone.
The next lemma is a slight generalization
of \cite[Proposition\,4.3]{KSTT12}.

\begin{lem}\rm\label{lem.product}
Let $N$ decompose into non-empty sets $R$ and $L$,
$r\in {\dv R}^{\caPR}$ and $l\in{\dv R}^{\caPL}$.
\begin{itemize}
\item[(i)] If both $r\in\caKRw$ and $l\in\caKLw$ 
then the function $m\in {\dv R}^{\caP}$ defined by
\begin{equation}
m(S) ~:=~ r(S\cap R)\cdot l(S\cap L)\qquad
\mbox{for any $S\subseteq N$,}
\label{eq.def-product}
\end{equation}
belongs to the cone $\caKw$. Conversely, if $m\in\caKw$ is non-zero and \eqref{eq.def-product}
holds with non-negative $r\in {\dv R}^{\caPR}$ and $l\in{\dv R}^{\caPL}$ then $r\in\caKRw$ and $l\in\caKLw$.
\item[(ii)] The function $m$ given by \eqref{eq.def-product} generates an extreme ray of $\caKw$
iff $r$ generates an extreme ray of $\caKRw$ and $l$ generates an the extreme ray of $\caKLw$.
\end{itemize}
\end{lem}

\begin{proof}
(i): in the non-trivial case of non-zero $r$ and $l$ one has $r(R)>0$, $l(L)>0$ and
the observation that $m$ is non-negative and non-decreasing is easy. Consider $\langle a,b|C\rangle\in\caE$
and distinguish three cases of its relation to $R$ and $L$. If $a\in R$ and $b\in L$ then write
\begin{eqnarray*}
\lefteqn{\hspace*{-5mm}\langle m,u_{\langle a,b|C\rangle}\rangle =
m(a\cup b\cup C) -m(a\cup C)-m(b\cup C)+m(C)}\\
&=&
r(a\cup  (C\cap R))\cdot l(b\cup (C\cap L)) - r(a\cup  (C\cap R))\cdot l(C\cap L)\\
&& -\,r(C\cap R)\cdot l(b\cup (C\cap L)) +r( C\cap R)\cdot l(C\cap L)\\
&=& [\, r(a\cup  (C\cap R)) - r(C\cap R)\,]\cdot [\, l(b\cup  (C\cap L)) - l(C\cap L)\,]\geq 0\,,
\end{eqnarray*}
because $r$ and $l$ are non-decreasing. If $a,b\in R$ then write
\begin{eqnarray*}
\lefteqn{\hspace*{-5mm}\langle m,u_{\langle a,b|C\rangle}\rangle =
m(a\cup b\cup C) -m(a\cup C)-m(b\cup C)+m(C)}\\
&=&
r(a\cup b\cup (C\cap R))\cdot l(C\cap L) - r(a\cup  (C\cap R))\cdot l(C\cap L)\\
&& -\,r(b\cup (C\cap R))\cdot l(C\cap L) +r( C\cap R)\cdot l(C\cap L)\\
&=& l(C\cap L)\cdot\langle r,u_{\langle a,b|C\cap R\rangle}\rangle \geq 0\quad
\mbox{since $l$ is non-negative and $r$ supermodular.}
\end{eqnarray*}
The case $a,b\in L$ is analogous, just interchange $R$ and $L$.

As concerns the converse claim in (i) note that $m(N)>0$ implies $r(R)>0$ and $l(L)>0$.
Because the function $A\subseteq R\mapsto m(A\cup L) \stackrel{\eqref{eq.def-product}}{=} r(A)\cdot l(L)$ is
non-decreasing with respect to inclusion, the same holds for $r$. Given $\langle a,b|C\rangle\in\caER$
observe
$$
l(L)\cdot\langle r,u_{\langle a,b|C\rangle}\rangle \stackrel{\eqref{eq.def-product}}{=} \langle m,u_{\langle a,b|C\cup L\rangle}\rangle\geq 0\qquad
\mbox{which implies that $r$ is supermodular.}
$$
The arguments for $l$ are analogous.
\\[0.4ex]
(ii): note that generators of extreme rays must be non-zero and $m\in\caKw$ generates
an extreme ray if $m=m^{1}+m^{2}$ for $m^{1},m^{2}\in\caKw$ implies that $m^{1}=\alpha\cdot m$ and
$m^{2}=\beta\cdot m$ with $\alpha ,\beta\geq 0$. Thus, the necessity of $r$ being extreme follows
from (i) because $r=r^{1}+r^{2}$ for $r^{1},r^{2}\in\caKRw$ implies $m=m^{1}+m^{2}$, where
$m^{i}$ is obtained from $r^{i}$ and $l$ by \eqref{eq.def-product}, $i=1,2$.
Thus, $m^{1}=\alpha\cdot m$ with $\alpha\geq 0$ implies $r^{1}=\alpha\cdot r$ for $l(L)>0$
and similarly with $r^{2}$. The necessity of
$l$  generating an extreme ray of $\caKLw$ is shown analogously.

As concerns the sufficiency assume that $m$ given by \eqref{eq.def-product} is written
$m=m^{1}+m^{2}$ with non-zero $m^{1},m^{2}\in\caKw$. First, we show that non-negative
$r^{1},r^{2}\in {\dv R}^{\caPR}$ exist such that
\begin{equation}
r^{1}+r^{2}=r ~~\&~~ m^{i}(S)=r^{i}(S\cap R)\cdot l(S\cap L),\quad \mbox{for $S\subseteq N$, $i=1,2$.}
\label{eq.proof-product}
\end{equation}
Given $A\subseteq R$ with $r(A)=0$ we put $r^{1}(A)=r^{2}(A)=0$, which works in the case
of any set $S\subseteq N$ with $S\cap R=A$ because then $m(S)=0$ and $m^{1}(S)=m^{2}(S)=0$.
Given $A\subseteq R$ with $r(A)>0$ consider the functions $B\subseteq L \mapsto m^{1}(A\cup B)$.
This non-negative and non-decreasing function is the minor of $m^{1}$ with deleting
of $R\setminus A$ and extracting of $A$, and, therefore, supermodular, by Observation \ref{obs.minor}.
Thus, its $\frac{1}{r(A)}$-multiple belongs to $\caKLw$.
Of course, the same conclusion holds in case of $B\subseteq L \mapsto m^{2}(A\cup B)$.
Now, realize that
\begin{eqnarray*}
\lefteqn{\hspace*{-3cm}\forall\, B\subseteq L\qquad r(A)\cdot l(B) \stackrel{\eqref{eq.def-product}}{=} m(A\cup B) =m^{1}(A\cup B)+m^{2}(A\cup B)}\\
&\Rightarrow& l(B)=\underbrace{\frac{1}{r(A)}\cdot m^{1}(A\cup B)}_{\in\caKLw} +\underbrace{\frac{1}{r(A)}\cdot m^{2}(A\cup B)}_{\in\caKLw}\,,
\end{eqnarray*}
and the assumption that $l$ generates an extreme ray of $\caKLw$ implies the existence of constants
$r^{1}(A),r^{2}(A)\geq 0$ such that $m^{i}(A\cup B)=r^{i}(A)\cdot l(B)$ for $B\subseteq L$, $i=1,2$. Then
$$
r(A)\cdot l(L) \stackrel{\eqref{eq.def-product}}{=} m(A\cup L) =m^{1}(A\cup L)+m^{2}(A\cup L)=
r^{1}(A)\cdot l(L)+ r^{2}(A)\cdot l(L)
$$
implies $r(A)=r^{1}(A)+r^{2}(A)$ because $l(L)>0$. Thus, this choice of $r^{1}(A)$ and $r^{2}(A)$,
ensures that \eqref{eq.proof-product} holds for $S\subseteq N$ with $S\cap R=A$.

The condition \eqref{eq.proof-product} says that non-zero $m^{1}$ is obtained from non-negative $r^{1}$ and $l$ by  \eqref{eq.def-product}.
Hence, by the second observation in (i), $r^{1}\in\caKRw$. Analogously,  $r^{2}\in\caKRw$ and
the assumption that $r$ generates an extreme ray of $\caKRw$ implies that $r^{1}=\alpha\cdot r$ and
$r^{2}=\beta\cdot r$ with $\alpha,\beta\geq 0$. Hence, \eqref{eq.proof-product} and \eqref{eq.def-product}
imply the desired conclusion about $m$, $m^{1}$ and $m^{2}$.
\end{proof}

To apply Lemma \ref{lem.product} realize what are the generators of extreme rays of $\caKw$. They break into three types, namely
the constant function, singleton superset identifiers, and the generators of extreme rays of the cone
$\caKl$ of $\ell$-standardized supermodular functions.

The composition with the constant function $m^{\emptyset\subseteq}$ is, in fact, the lifting transformation
discussed already in \S\,\ref{ssec.lift-tran}. The next corollary is devoted to the composition of an
extreme $\ell$-standardized function with a singleton subset identifier;
it leads to a linear transformation discussed later in \S\,\ref{ssec.lower-modular}.
The observation has already been mentioned in \cite[Lemma\,4.4]{KSTT12}.

\begin{cor}\rm\label{cor.prod-modular} ~
Assume $|N|\geq 2$, $x\in N$ and $M:= N\setminus x$. Given $r\in\caKMl$, let $m^{\prime}\in {\dv R}^{\caP}$ be defined
as follows:
\begin{equation}
m^{\prime}(S) := \left\{\
\begin{array}{cl}
r(S\cap M) & \mbox{if $x\in S$,}\\
0 & \mbox{otherwise,}
\end{array}
\right.
\qquad \mbox{for any $S\subseteq N$}.
\label{eq.prod-modular}
\end{equation}
Then $m^{\prime}\in\caKl$ and, moreover, $r$ is extreme in $\caKMl$ iff $m^{\prime}$ is extreme in $\caKl$.
\end{cor}

\begin{proof}
Apply Lemma \ref{lem.product} with $R=M$, $L=x$ and $l=m^{x\subseteq}$. Then $m$ defined by
\eqref{eq.def-product} coincides with $m^{\prime}$ given by \eqref{eq.prod-modular} and is evidently $\ell$-standardized.
The rest follows from Lemma \ref{lem.product}.
\end{proof}
\smallskip

The product-based combination of two $\ell$-standardized supermodular functions is treated in another
corollary, which is, basically, the result from \cite[Proposition\,4.3]{KSTT12}.

\begin{cor}\rm\label{cor.prod-general} ~
Let $N$ decompose into non-empty sets $R$ and $L$.
Assume that $r\in\caKRl$, $l\in\caKLl$ and the function $m\in {\dv R}^{\caP}$ is defined
by \eqref{eq.def-product}. Then $m\in\caKl$, and, moreover, $m$ is extreme in $\caKl$
iff [\,$r$ is extreme in $\caKRl$ and $l$ extreme in $\caKLl$\,].
\end{cor}

\begin{proof}
This follows directly from Lemma \ref{lem.product}, taking into account
what was said above.
\end{proof}
\smallskip

Note that the product-based composition as a tool for generating ``new"
types of extreme supermodular functions only applies in the case $|N|\geq 4$; this
is because non-zero extreme $\ell$-standardized supermodular functions exist if they are at least two variables. In the case $|N|=4$ the result of product-based composition is equivalent to the application of the lower modular extension from \S\,\ref{ssec.lower-modular}.
However, in the case $|N|\geq 5$ its application results in non-trivial types
of extreme supermodular functions.

The specialty of this operation is that,
if applied to extreme supermodular functions with zero-one table of scalar products
(see Definition \ref{def.CI-induced}) it may result in an extreme supermodular function whose table of scalar
products has more than two distinct values. None of other transformations
discussed in this report has this property.

\subsection{Lower modular extension}\label{ssec.lower-modular}
The extension of $\ell$-standardized functions mentioned in Corollary \ref{cor.prod-modular}
is a special case of the following linear transformation.

\begin{defin}\rm\label{def.lower-exten}
Let $M,N$ be sets of variables such that $\emptyset\neq M\subset N$.
The {\em lower modular extension\/} transformation  $\lowmod : {\dv R}^{\caPM}\to {\dv R}^{\caP}$ is
defined as follows:
\begin{eqnarray}
\lefteqn{\hspace*{-1.6cm}\lowmod :
r\in {\dv R}^{\caPM}\mapsto\, \lowmod (r)\equiv m\in {\dv R}^{\caP},~~\mbox{where}} \nonumber\\[0.2ex]
m(S) &:=& \left\{\
\begin{array}{cl}
r(S\cap M) & \mbox{if $(N\setminus M)\subseteq S$,}\\
r(S\cap M)-r_{\ell}(S\cap M) & \mbox{otherwise,}
\end{array}
\right.
\quad \mbox{for any $S\subseteq N$}.
\label{eq.low-exten-def}
\end{eqnarray}
\end{defin}
\bigskip

Observe that the difference $r-r_{\ell}\in {\dv R}^{\caPM}$ is a modular function: specifically, \eqref{eq.l-stan} says
$$
r-r_{\ell}=
r(\emptyset)\cdot m^{\emptyset\subseteq}+\sum_{i\in M} \,[r(i)-r(\emptyset)]\cdot m^{i\subseteq}\in \caLM\
$$
and one can interpret $r-r_{\ell}$ as the lower modular part of $r$. This perhaps explains the motivation for the term {\em lower modular extension}: one basically ``copies" $r$ to the ``upper" class of sets
$(N\setminus M)\cup\caPM :=\{\, (N\setminus M)\cup L\,:\ L\subseteq M\}$
and extends this copy of $r$ by the copies of its lower modular part ``written" to all remaining ``lower" classes $R\cup\caPM$, $R\subset (N\setminus M)$.
In particular, the minor of $m=\lowmod (r)$ with extracting of $N\setminus M$ is again $r$, which means that the extraction minor projection from \S\,\ref{ssec.minors} is a complementary operation to $\lowmod$.
Note that $\lowmod$ maps linearly $\caLM$ to $\caL$, specifically
$$
\alpha_{\emptyset}\cdot m^{\emptyset\subseteq}
+\sum_{i\in M} \,\alpha_{i}\cdot m^{i\subseteq}\in\caLM
~\stackrel{\lowmod}{\longmapsto}~
\alpha_{\emptyset}\cdot m^{\emptyset\subseteq}
+\sum_{i\in M} \,\alpha_{i}\cdot m^{i\subseteq}\in\caL
\quad \mbox{for $\alpha_{\emptyset},\alpha_{i}\in {\dv R}$},
$$
where the pre-image is a function in ${\dv R}^{\caPM}$ while the image,
given by the same formal expression, is viewed as a function in ${\dv R}^{\caP}$.
Hence, $r^{1}\approx r^{2}$ implies
$\lowmod (r^{1})\approx \lowmod (r^{2})$ and the lower modular extension is a transformation between equivalence classes of\, $\approx$.

Without loss of generality, one can limit oneself to the case
$|N\setminus M|=1$ because repeated lower modular extension is equivalent to
lower modular extension with a larger set $N\setminus M$.
Indeed, if $N\setminus M=\{ x_{1},x_{2}\}$ then one first applies $\lowmod$
with $N\setminus M=x_{1}$ to $r$ and gets $r^{\prime}$ on ${\cal P}(M^{\prime})$ where $M^{\prime}=M\cup x_{1}$. Observe that the lower modular part of $r^{\prime}$
can be interpreted as two copies of the lower modular part of $r$
``written" to $\caPM$ and $x_{1}\cup\caPM$. Therefore, the application of
$\lowmod$ with $N\setminus M^{\prime}=x_{2}$ to $r^{\prime}$
has then the same effect as the application of $\lowmod$ to $r$ with $N\setminus M=\{x_{1},x_{2}\}$.

\begin{obser}\label{obs.lower-exten}\rm  ~~
Assume $|N|\geq 2$, $x\in N$ and $M:= N\setminus x$ and $r\in {\dv R}^{\caPM}$.
\begin{itemize}
\item[(i)]
One has $m=\lowmod (r)\in\caK$ iff $r\in\caKM$. The independency model
induced by $\lowmod (r)$ is then determined as follows: given $\langle a,b|C\rangle\in\caE$
one has
\begin{equation}
a\ci b\,|\,C\,\,[m] ~~\Leftrightarrow ~~
\left\{
\begin{array}{l}
\mbox{either}~~ x\not\in (a\cup b\cup C)\,,\\
\mbox{or}~~  x\in\{ a,b\} ~\&~~ r_{\ell}(\,(a\cup b\cup C)\setminus x)=r_{\ell}(C)\,,\\
\mbox{or}~~  x\in C ~\&~ a\ci b\,|\, (C\setminus x)\,\, [r]\,.
\end{array}
\right.
\label{eq.lower-exten}
\end{equation}
\item[(ii)] The function $r$ is extreme in  $\caKM$ iff $m=\lowmod (r)$ is extreme in $\caK$.
\end{itemize}
\end{obser}

\begin{proof}
(i): Observation \ref{obs.minor} implies the necessity of $r=T_{|x}(m)\in\caKM$ for $m\in\caK$.
As concerns the sufficiency realize that one can re-write \eqref{eq.low-exten-def} in the form
$$
m(S) = \left\{\
\begin{array}{cl}
r(\emptyset) +\sum_{i\in S\cap M} [r(i)-r(\emptyset)] +r_{\ell}(S\cap M)& \mbox{if $x\in S$,}\\
r(\emptyset) +\sum_{i\in S\cap M} [r(i)-r(\emptyset)] & \mbox{if $x\not\in S$,}
\end{array}
\right.
\quad \mbox{for $S\subseteq N$},
$$
which means $m=l+m^{\prime}$, where
$l=r(\emptyset)\cdot m^{\emptyset\subseteq}+\sum_{i\in M} \,[r(i)-r(\emptyset)]\cdot m^{i\subseteq}\in \caL$
and $m^{\prime}=\lowmod (r_{\ell})$ is the ``lower" zero extension of $r_{\ell}$ defined by \eqref{eq.prod-modular}
in Corollary \ref{cor.prod-modular}. This allows us to deduce \eqref{eq.lower-exten}: since $m_{\ell}\approx m^{\prime}$ one has
$a\ci b\,|\,C\,[m]$ iff $a\ci b\,|\,C\,\, [m^{\prime}]$ for any $\langle a,b|C\rangle\in\caE$. Moreover,
Corollary \ref{cor.prod-modular} says $ m^{\prime}\in\caKl$, which implies $m\in\caK$.\\[0.4ex]
(ii):  Corollary \ref{cor.isom-stan} says that the extremality in $\caK$ reduces to the extremality in $\caKl$ and Corollary \ref{cor.prod-modular} claims that $r_{\ell}$ is extreme in $\caKMl$ iff its extension $m^{\prime}$ is extreme in $\caKl$.
\end{proof}

\begin{example}\label{exa.lower-exten}
Take $N=\{ a,b,c,d\}$, $M=\{ a,b,c\}$ and $\bar{r}=\delta_{\{ a,b,c\}}\in {\dv R}^{\caPM}$, which is extreme
in $\caKM$. As $\bar{r}$ itself is $\ell$-standardized, its lower modular extension
$m=\delta_{\{ a,b,c,d\}}\in {\dv R}^{\caP}$ is also $\ell$-standardized. Observation \ref{obs.lower-exten}(ii) says
$m$ is extreme in $\caK$. The extremality of $m$ also follows from Corollary \ref{cor.prod-general}
with $R=\{ a,b\}$, $L=\{ c,d\}$ , $r=\delta_{\{ a,b\}}$ and $l=\delta_{\{ c,d\}}$.
\end{example}

\subsection{Upper modular extension}\label{ssec.upper-modular}
One can combine the lower modular extension from \S\,\ref{ssec.lower-modular} with
the reflection transformation from \S\,\ref{ssec.reflection}, which leads to the following concept.

\begin{defin}\rm\label{def.upper-exten}
Let $M,N$ be sets of variables such that $\emptyset\neq M\subset N$.
The {\em upper modular extension\/} transformation  $\uppmod : {\dv R}^{\caPM}\to {\dv R}^{\caP}$ is
defined as follows:
\begin{eqnarray}
\lefteqn{\hspace*{-1.6cm}\uppmod :
r\in {\dv R}^{\caPM}\mapsto\, \uppmod (r)\equiv m\in {\dv R}^{\caP},~~\mbox{where}} \nonumber\\[0.2ex]
m(S) &:=& \left\{\
\begin{array}{cl}
r(S) & ~~\mbox{if $S\subseteq M$,}\\
r(S\cap M)-r_{u}(S\cap M) & ~~\mbox{otherwise,}
\end{array}
\right.
\quad \mbox{for any $S\subseteq N$}.
\label{eq.upp-exten-def}
\end{eqnarray}
\end{defin}
\bigskip

The difference $r-r_{u}\in {\dv R}^{\caPM}$ is a modular function and can be interpreted as an upper modular part of $r$.
Thus, the term {\em upper modular extension\/} has similar motivation as the term lower modular extension from \S\,\ref{ssec.lower-modular}.
The formula \eqref{eq.upp-exten-def} can be interpreted as follows:
one ``copies" $r$ to $\caPM$ and then extends this copy of $r$ by the copies
of its upper modular extension $r-r_{u}$ ``written" to all remaining classes
$R\cup\caPM$, where $\emptyset\neq R\subseteq N\setminus M$.
As concerns the expression for $r-r_{u}$, one has
$$
r(S\cap M)-r_{u}(S\cap M) \stackrel{\eqref{eq.u-stan}}{=}
\sum_{i\in M\setminus S} r(M\setminus i) + (1-|M\setminus S|)\cdot r(M)\quad
\mbox{for any $S\subseteq N$}.
$$
It follows from \eqref{eq.upp-exten-def} that the restriction of $m=\uppmod (r)$
to $\caPM$ is again $r$; hence, the deletion minor projection $T_{-(N\setminus M)}$ is a complementary operation to $\uppmod$.

Because $m^{\emptyset\subseteq},m^{j\subseteq}\in {\dv R}^{\caPM}$, for $j\in M$,
are ascribed $m^{\emptyset\subseteq},m^{j\subseteq}\in {\dv R}^{\caP}$ by $\uppmod$,
the transformation maps linearly $\caLM$ to $\caL$. This says $r^{1}\approx r^{2} ~\Rightarrow~ \uppmod (r^{1})\approx \uppmod (r^{2})$
and the upper modular extension can be interpreted as
a transformation of $\approx$-equivalence classes.
Without loss of generality, one can limit oneself to the case
$|N\setminus M|=1$ since repeated upper modular extension is equivalent to
upper modular extension with a larger set $N\setminus M$; the arguments are
analogous to those used in the case of lower modular extension.

The relation of $\uppmod$ and $\lowmod$ can be described as follows. The
application of the upper modular extension $\uppmod$ followed by the application of the reflection transformation on ${\dv R}^{\caP}$ (see \S\,\ref{ssec.reflection}) has the same effect as the application of the reflection transformation on ${\dv R}^{\caPM}$ followed by the application of the lower modular transformation $\lowmod$.
Indeed, this observation follows from the above descriptions of both transformations
in terms of ``copies" to classes $R\cup\caPM$, $R\subseteq N\setminus M$ and the
formula $T_{\iota}(r-r_{u})= T_{\iota}(r)- T_{\iota}(r_{u}) =
T_{\iota}(r)- (T_{\iota}(r))_{\ell}$; see the remark concluding \S\,\ref{ssec.reflection}.
An analogous statement holds with interchanged $\uppmod$ and $\lowmod$, that is,
one has
\begin{equation}
\forall\,r\in {\dv R}^{\caPM}\qquad
T_{\iota}(\uppmod (r)) =\lowmod (T_{\iota}(r)) ~~\&~~ T_{\iota}(\lowmod (r)) =\uppmod (T_{\iota}(r))\,.
\label{eq.dual-modular}
\end{equation}
Owing to the relation of $\uppmod$ and $\lowmod$ the following observation
can easily be derived from Observation \ref{obs.lower-exten}; a special case
of it has been shown in \cite[Lemma\,4.5]{KSTT12}.

\begin{obser}\label{obs.upper-exten}\rm  ~~
Assume $|N|\geq 2$, $x\in N$ and $M:= N\setminus x$ and $r\in {\dv R}^{\caPM}$.
\begin{itemize}
\item[(i)]
One has $m=\uppmod (r)\in\caK$ iff $r\in\caKM$. The independency model
induced by $\uppmod (r)$ is then determined as follows: given $\langle a,b|C\rangle\in\caE$
one has
\begin{equation}
a\ci b\,|\,C\,\,[m] ~~\Leftrightarrow ~~
\left\{
\begin{array}{l}
\mbox{either}~~ x\in C\,,\\
\mbox{or}~~  x\in\{ a,b\} ~\&~~ r_{u}(\,(a\cup b\cup C)\setminus x)=r_{u}(C)\,,\\
\mbox{or}~~  x\not\in (a\cup b\cup C) ~\&~ a\ci b\,|\, C\,\, [r]\,.
\end{array}
\right.
\label{eq.upper-exten}
\end{equation}
\item[(ii)] The function $r$ is extreme in  $\caKM$ iff $m=\uppmod (r)$ is extreme in $\caK$.
\end{itemize}
\end{obser}

\begin{proof}
(i): Observation \ref{obs.reflect}(i) with $M$ in place of $N$ says $r\in\caKM$ iff
$r\compo\boliota= T_{\iota}(r)\in\caKM$.
The latter is equivalent to $T_{\iota}(\uppmod (r))\stackrel{\eqref{eq.dual-modular}}{=}\lowmod (T_{\iota}(r))\in\caK$ by Observation \ref{obs.lower-exten}(i).
Observation \ref{obs.reflect}(i) implies the latter is equivalent to $\uppmod (r)\in\caK$.
To derive \eqref{eq.upper-exten} denote $m^{\prime}:=\uppmod (r_{u})$, observe that
$r\approx r_{u} ~\Rightarrow~ m=\uppmod (r)\approx \uppmod (r_{u})=m^{\prime} ~\Rightarrow~ m\sim m^{\prime}$
and realize that $m^{\prime}$ vanishes outside $\caPM$ and coincides with $r_{u}$ within $\caPM$.\\[0.4ex]
(ii): the arguments are analogous to the case of (i): one applies Observation \ref{obs.reflect}(ii) with
$M$ in place of $N$, then Observation \ref{obs.lower-exten}(ii), \eqref{eq.dual-modular} and, finally,
Observation \ref{obs.reflect}(ii) with $N$.
\end{proof}

\begin{example}\label{exa.lower-exten}
Take $N=\{ a,b,c\}$, $M=\{ a,b\}$ and $\hat{r}=\delta_{\{ a,b\}}\in {\dv R}^{\caPM}$, which is extreme in $\caKM$.
One has $\hat{r}-\hat{r}_{u}=\delta_{\{ a,b\}} -\delta_{\emptyset}$ and the application of \eqref{eq.upp-exten-def} gives
$$
m=\uppmod (\hat{r}) = \delta_{\{a,b\}}  +\delta_{\{ a,b,c\}} -\delta_{c}\quad
\mbox{and}\quad
m_{\ell}= 2\cdot\delta_{\{a,b,c\}}  +\delta_{\{ a,b\}}  +\delta_{\{a,c\}}  +\delta_{\{ b,c\}}\,.
$$
Note that $\hat{r}_{u}=\delta_{\emptyset}\in {\dv R}^{\caPM}$ and $\uppmod (\hat{r}_{u})=\delta_{\emptyset}\in {\dv R}^{\caP}$.
Indeed, since $r\approx r_{u}=\delta_{\emptyset}$  one has
$\uppmod (r)\approx \uppmod (\delta_{\emptyset})$ and their $\ell$-standardizations
coincide: $m_{\ell}=(\delta_{\emptyset})_{\ell}$. This illustrates that $\uppmod$ for $u$-standardized functions
is just the ``upper" zero extension.
\end{example}

\section{Replications}\label{sec.replica}
In this section we discuss two other linear mappings to a higher-dimensional space
which, under certain non-restrictive conditions, preserve both supermodularity and extremality.
They cannot be interpreted directly as mappings between $\approx$-equivalence classes; however, after
minor modification, they can be viewed as mappings between $\approx$-equivalence classes
of {\em extreme\/} supermodular functions.

\subsection{Lower replication}\label{ssec.lower-replica}

\begin{defin}\rm\label{def.lower-replica}
Let $M,N$ be sets of variables such that $z\in M$ exists with $L:=M\setminus z\subset N$.
The {\em lower replication\/} transformation $\lowrepl : {\dv R}^{\caPM}\to {\dv R}^{\caP}$ is
defined as follows:
\begin{eqnarray}
\lefteqn{\hspace*{-1.8cm}\lowrepl:
r\in {\dv R}^{\caPM}\mapsto\, \lowrepl (r)\equiv m\in {\dv R}^{\caP},~~\mbox{where}} \nonumber\\[0.2ex]
m(S)&=&\left\{
\begin{array}{cl}
r(z\cup (S\cap L)) & \mbox{~~if~ $(N\setminus L)\subseteq S$,}\\
r(S\cap L) & \mbox{~~if~ $N\setminus (L\cup S)\neq\emptyset$,}\\
\end{array}
\right.
\qquad \mbox{for any $S\subseteq N$.}
\label{eq.low-replica-def}
\end{eqnarray}
\end{defin}
\medskip

To explain the motivation for the terminology ``decompose" virtually $r\in {\dv R}^{\caPM}$ into
its ``lower" part, which is its restriction to $\caPL$, and its ``upper" part,
which is its restriction to the class of sets $z\cup\caPL :=\{ z\cup R\,:\ R\subseteq L\}$.
In the non-trivial case $|N\setminus L|\geq 2$, the function $m=\lowrepl (r)$ is obtained by
copying once the ``upper" part of $r$ to $(N\setminus L)\cup\caPL$ and
by multiple copying  its ``lower" part (= the replication) to each $R\cup\caPL$, $R\subset N\setminus L$.

A simple observation is that, without loss of generality, one can limit oneself to the case
$|N\setminus L|=2$ because repeated lower replication is equivalent to
 lower replication with a larger set $N\setminus L$.
Indeed, if $N\setminus L=\{z_{1},z_{2},z_{3}\}$ then one first applies $\lowrepl$
with $N\setminus L=\{z_{1},z_{2}\}$ and gets $r^{\prime}=\lowrepl (r)$ on ${\cal P}(L\cup z_{1}\cup z_{2})$.
The choice $z^{\prime}=z_{2}$, $L^{\prime}=L\cup z_{1}$ and $N\setminus L^{\prime}=\{ z_{2},z_{3}\}$
has then the same effect as the application of $\lowrepl$ to $r$ with $N\setminus L=\{z_{1},z_{2},z_{3}\}$.

The lower replication can be interpreted as a simultaneous inversion of
some ``coinciding" projections. Specifically, if $N\setminus L=\{ x,y\}$
and $m=\lowrepl (r)$ then $r$ can be ``obtained" from $m$ in three ways:
taking the extraction minor $T_{|x}(m)$ and re-naming $y$ to $z$,
taking the extraction minor $T_{|y}(m)$ and re-naming $x$ to $z$, and
taking the contraction $T_{\{x,y\}\to x}(m)$ and re-naming $x$ to $z$,
respectively by renaming $y$ to $z$ in $T_{\{x,y\}\to y}(m)$.
The supermodular  extremality is not preserved by these projections, but,
it is preserved if they ``coincide".

\begin{obser}\label{obs.lower-repli}\rm  ~~
Let $N,M$ be sets of variables with $z\in M$, $L:=M\setminus z$ and $N\setminus L= \{ x,y\}$.
Assume $r\in {\dv R}^{\caPM}$.
\begin{itemize}
\item[(i)]
One has $m=\lowrepl (r)\in\caK$ iff [\,$r\in\caKM$ and $r(z)\geq r(\emptyset )$\,]. The independency model
induced by $\lowrepl (r)$ is then determined as follows: given $\langle a,b|C\rangle\in\caE$ one has
$a\ci b\,|\,C\,\,[m]$ in the following cases:
\begin{eqnarray*}
a\cap L\ci b\cap L\,|\, C\cap L\,\, [r] &\mbox{in case}& \{x,y\}\setminus (a\cup b\cup C)\neq\emptyset\,,\\
a\ci b\,|\, z\cup (C\cap L)\,\, [r] &\mbox{in case}& \{x,y\}\subseteq C,\\
z\ci \{a,b\}\cap L\,|\, C\cap L\,\, [r] &\mbox{in case}& |\{x,y\}\cap C|=1 ~\&~ \{x,y\}\subseteq a\cup b\cup C,\\
r(z\cup C)=r(C) &\mbox{in case}& \{x,y\}=\{ a,b\}\,.
\end{eqnarray*}
\item[(ii)] If $r\in\calK_{\ell}(M)$ then $\lowrepl (r)\in\caKl$.
A function $r\in\calK_{\ell}(M)$ is extreme in  $\caKM$ iff $\lowrepl (r)$ is extreme in $\caK$.
\item[(iii)] If $r$ is extreme in $\caKM$ and $r(z)=r(0)$ then $\lowrepl (r)$ is extreme in $\caK$
and, moreover, $\lowrepl (r)\approx\lowrepl (r_{\ell})$. Conversely, if $\lowrepl (r)$ is extreme
in $\caK$ then either the above case occurs or one has $r(z)>r(\emptyset )$, $r\in\caLM$
and $\lowrepl(r)\approx [r(z)-r(\emptyset)]\cdot m^{\{x,y\}\subseteq}$.
\end{itemize}
\end{obser}

\begin{proof}
(i): the necessity of $r\in\caKM$ is easy for $r$ can be interpreted as the minor
of $m$ obtained by extracting of $x$ and renaming $y$ to $z$ (Observation \ref{obs.minor}).
Owing to supermodularity
$$
0\leq \langle m,u_{\langle x,y|\emptyset\rangle}\rangle =
m(x\cup y) -m(x)-m(y)+m(\emptyset) = r(z)-r(\emptyset)-r(\emptyset)+r(\emptyset)=
r(z)-r(\emptyset)\,.
$$
As concerns the sufficiency, the first step is to realize that $r(z)\geq r(\emptyset)$ and
supermodularity of $r$ allows one to observe by induction on $|C|$ that
$r(z\cup C)\geq r(C)$ for any $C\subseteq L$.
The second step is to consider $\langle a,b|C\rangle\in\caE$ and distinguish the cases
of its relation to $\{ x,y\}$. If
$\{x,y\}\setminus (a\cup b\cup C)\neq\emptyset$
then write using \eqref{eq.low-replica-def}:
\begin{eqnarray*}
\lefteqn{\hspace*{-5mm}\langle m,u_{\langle a,b|C\rangle}\rangle =
m(a\cup b\cup C) -m(a\cup C)-m(b\cup C)+m(C)}\\
&=&
r(\,(a\cup b\cup C)\cap L\,) -r(\,(a\cup C)\cap L\,)-r(\,(b\cup C)\cap L\,)+r(C\cap L)\\
&=& \langle r,u_{\langle a\cap L,b\cap L\,|\,C\cap L\rangle}\rangle\geq 0\quad \mbox{because $r\in\caKM$.}
\end{eqnarray*}
If $\{ x,y\}\subseteq a\cup b\cup C$ then, depending on $|\{ x,y\}\cap C|$, three subcases can occur.
In the subcase $x,y\in C$ one gets $\langle m,u_{\langle a,b|C\rangle}\rangle=
\langle r,u_{\langle a,b|\tilde{C}\cup z\rangle}\rangle\geq 0$, where $\tilde{C}=C\setminus \{x,y\}$. If $x=a$ and $y\in C$ then
denote $C^{\prime}=C\setminus y$ and write
\begin{eqnarray*}
\lefteqn{\hspace*{-10mm}\langle m,u_{\langle a,b|C\rangle}\rangle =
m(x\cup b\cup C) -m(x\cup C)-m(b\cup C)+m(C)}\\
&=&
r(z\cup b\cup C^{\prime}\,) -r(z\cup C^{\prime}\,) -r(b\cup C^{\prime}\,) +r(C^{\prime})
= \langle r,u_{\langle z,b|C^{\prime}\rangle}\rangle\geq 0\,.
\end{eqnarray*}
An analogous reasoning works for other variants of $|\{ x,y\}\cap C|=1$. Finally, in the
last subcase $\{ x,y\}=\{ a,b\}$ write
\begin{eqnarray*}
\lefteqn{\hspace*{-15mm}\langle m,u_{\langle a,b|C\rangle}\rangle =
m(x\cup y\cup C) -m(x\cup C)-m(y\cup C)+m(C)}\\
&=&
r(z\cup C\,) -r(C) -r(C) +r(C) =r(z\cup C\,) -r(C)\geq 0\,,
\end{eqnarray*}
by the first step. Thus, $m\in\caK$ has been verified.
The form of the independency model induced by $m$ follows from the proof, which breaks into the above mentioned cases.\\[0.4ex]
(ii): the fact that $\lowrepl$ maps $\ell$-standardized functions to $\ell$-standardized functions is evident;
thus, the first claim follows from (i). As concerns the second claim, by Corollary \ref{cor.isom-stan},
it is enough to show that $r$ generates an extreme ray of $\caKMl$ iff
$\lowrepl (r)$ generates an extreme ray of $\caKl$. For the necessity of the latter realize
that any $m\in\caKl$ is non-decreasing, for which reason
$$
\{ m\in\caKl \,:~  m (x\cup y)=0 ~~~\&~~~ m(x\cup R)=m(R)= m(y\cup R)~~\mbox{for any $R\subseteq L$}\,\}
$$
is a face of $\caKl$. One can easily observe that this face is nothing but the image of $\caKMl$ by $\lowrepl$.
Thus, given $r\in {\dv R}^{\caPM}$ generating an extreme ray of $\caKMl$,
if $m=\lowrepl (r)$ is an inner point of the segment between $m^{1}$ and $m^{2}$ in $\caKl$ then,
since $m$ belongs to that face, $m^{1},m^{2}$ do so. In particular, there are $r^{1},r^{2}\in\caKMl$
such that $\lowrepl (r^{i})=m^{i}$ for $i=1,2$ and, by linearity of the inversion of $\lowrepl$,
$r$ is an inner point of the segment between $r^{1}$ and $r^{2}$.
Thus, $r^{i}$ are non-negative multiples of $r$ and $m^{i}$
must be non-negative multiples of $m$; the fact that $m=\lowrepl (r)$ generates an extreme ray of
$\caKl$ has been verified. The sufficiency of the latter condition for $r$ generating an extreme
ray of $\caKMl$ can be shown analogously.\\[0.4ex]
(iii): as concerns the first claim, realize that in case $r(z)=r(\emptyset )$ the formula
\eqref{eq.l-stan} implies that $r-r_{\ell}$ is a linear combination of $m^{\emptyset\subseteq}$
and $m^{i\subseteq}$ for $i\in L$. Since these functions are mapped by $\lowrepl$ to the functions
of the same form on $\caP$, $\lowrepl (r)-\lowrepl (r_{\ell})$ is a linear combination of these,
$\lowrepl (r)\approx\lowrepl (r_{\ell})$ and the first claim follows from (ii) using Corollary~\ref{cor.isom-stan}.
As concerns the converse claim, if $r(z)=r(\emptyset)$, use the same reasoning and the converse
implication in~(ii). In case $\alpha:=r(z)-r(\emptyset)\neq 0$ write $r=\alpha\cdot m^{z\subseteq}+r^{\prime}$,
where $r^{\prime}\approx r$ satisfies $r^{\prime}(z)=r^{\prime}(\emptyset )$.
Observe $\lowrepl (r)\in\caK ~\stackrel{\mbox{\tiny (i)}}{\Rightarrow}~ r\in\caKM$
and, necessarily $\alpha = r(z)-r(\emptyset)>0$. Then $r\in\caKM
~\Leftrightarrow~ r^{\prime}\in\caKM  ~\stackrel{\mbox{\tiny (i)}}{\Rightarrow}~
\lowrepl (r^{\prime})\in\caK$. As $\lowrepl$ maps $m^{z\subseteq}$ to $m^{\{ x,y\}\subseteq}$ one has
$\lowrepl (r)=\alpha\cdot m^{\{ x,y\}\subseteq} +\lowrepl (r^{\prime})$, where $m^{\prime}:=\lowrepl (r^{\prime})$
is not a non-zero multiple of $m^{\{ x,y\}\subseteq}$ since $m^{\prime}(\{ x,y\})=m^{\prime}(\emptyset )$.
Thus, provided $\lowrepl (r)$ is extreme in $\caK$, the function $\lowrepl (r^{\prime})=m^{\prime}$ must be modular.
Because $m^{x\subseteq}$ and $m^{y\subseteq}$ are not in the image of $\lowrepl$, this means $m^{\prime}$ is
a linear combination of $m^{\emptyset\subseteq}$ and $m^{i\subseteq}$ for $i\in L$.
Hence, the same holds for $r^{\prime}$ on $\caPM$ and $r$ must belong to $\caLM$.
Since $\lowrepl (r^{\prime})\in\caL$ one has
$\lowrepl (r)\approx \alpha\cdot  m^{\{ x,y\}\subseteq}$.
\end{proof}

The following consequence allows one to interpret the lower replication as a mapping between
$\approx$-equivalence classes of extreme supermodular functions.

\begin{cor}\label{cor.lower-repli}\rm  ~~
In the situation from Observation \ref{obs.lower-repli},
if $r$ is extreme in $\caKM$ then
$$
\lowrepl (r)\approx \lowrepl (r_{\ell})~~\Leftrightarrow~~ \lowrepl (r) ~~\mbox{is extreme in $\caK$} ~~\Leftrightarrow~~ r(z)=r(\emptyset)\,.
$$
\end{cor}

\begin{proof}
Corollary \ref{cor.isom-stan} implies that the extremality of a function and of its $\ell$-standardized version are equivalent. Since $r_{\ell}\approx r$ is extreme in $\caKM$, by Observation \ref{obs.lower-repli}(ii),
$\lowrepl (r_{\ell})$ is extreme in $\caK$. Hence, $\lowrepl (r)\approx \lowrepl (r_{\ell})$ implies the extremality of $\lowrepl (r)$ in $\caK$.
If $\lowrepl (r)$ is extreme in $\caK$ then, by Observation \ref{obs.lower-repli}(i), $r(z)\geq r(\emptyset)$. The case $r(z)-r(\emptyset )>0$ cannot
occur because then, by Observation \ref{obs.lower-repli}(iii), $r\in\caLM$, which is a contradiction with
the assumption that $r$ is extreme in $\caKM$. Therefore, $r(z)=r(\emptyset )$, which implies that $\lowrepl (r)\approx \lowrepl (r_{\ell})$, again by Observation \ref{obs.lower-repli}(iii).
The equivalence of the extremality of $\lowrepl (r)$ in $\caK$ with $r(z)=r(\emptyset)$ follows analogously from Observation \ref{obs.lower-repli}(iii).
\end{proof}

To summarize: any extreme supermodular function $r$ on $\caPM$ produces, through the lower replication,
an extreme supermodular function on $\caP$. Of course, $\lowrepl (r)$ itself need not be supermodular
and, even if supermodular, it need not be extreme in $\caK$. However, one can always take $r_{\ell}$
in place of $r$ and, by Observation \ref{obs.lower-repli}(ii), $\lowrepl (r_{\ell})$ is ensured to be extreme
in $\caK$. Moreover,  by Corollary \ref{cor.lower-repli}, if $r\approx r^{\prime}\in\caKM$ such that $\lowrepl (r^{\prime})$ is extreme in $\caK$ then $\lowrepl (r^{\prime})\approx \lowrepl (r_{\ell})$. Thus, the $\approx$-equivalence class of extreme functions in $\caK$ ascribed to the $\approx$-equivalence class of $r$ is uniquely determined!

The technical condition $r(z)=r(\emptyset )$, under which $\lowrepl (r)$ is guaranteed to be extreme (see Corollary \ref{cor.lower-repli}),
is not restrictive because, without loss of generality, one can limit oneself to the $\ell$-standardized framework.
\medskip

\begin{example}\label{exa.lower-repli}
Consider $N=\{a,b,c,d\}$, $M=\{ a,b,c\}$ and $z=c$, that is, $L=\{ a,b\}$. Take
$r=2\cdot\delta_{\{ a,b,c\}}+ \delta_{\{ a,b\}}+\delta_{\{ a,c\}}+\delta_{\{ b,c\}}\in\caKMl$.
Then the application of \eqref{eq.low-replica-def} gives
$$
m= \lowrepl (r)= 2\cdot\delta_{\{ a,b,c,d\}}+ \delta_{\{ a,b,c\}}+ \delta_{\{ a,b,d\}}+\delta_{\{ a,c,d\}}+\delta_{\{ b,c,d\}}+ \delta_{\{ a,b\}}\,,
~~\mbox{which is,}
$$
by Observation \ref{obs.lower-repli}(ii), an extreme $\ell$-standardized supermodular function on $\caP$.
\end{example}

\subsection{Upper replication}\label{ssec.upper-replica}
The lower replication from \S\,\ref{ssec.lower-replica} can be combined with
the reflection transformation from \S\,\ref{ssec.reflection}, which leads to the following concept.

\begin{defin}\rm\label{def.upper-replica}
Let $M,N$ be sets of variables such that $z\in M$ exists with $L:=M\setminus z\subset N$.
The {\em upper replication\/} transformation $\upprepl : {\dv R}^{\caPM}\to {\dv R}^{\caP}$ is
defined as follows:
\begin{eqnarray}
\lefteqn{\hspace*{-1.8cm}\upprepl:
r\in {\dv R}^{\caPM}\mapsto\,  \upprepl(r)\equiv m\in {\dv R}^{\caP},~~\mbox{where}} \nonumber\\[0.2ex]
m(S)&=&\left\{
\begin{array}{cl}
r(z\cup (S\cap L)) & \mbox{~~if~ $S\setminus L\neq\emptyset$,}\\
r(S) & \mbox{~~if~ $S\subseteq L$,}\\
\end{array}
\right.
\qquad \mbox{for any $S\subseteq N$.}
\label{eq.upp-replica-def}
\end{eqnarray}
\end{defin}
\medskip

It can be derived easily from the definition that the
application of the upper replication $\upprepl$ followed by the application of the
reflection transformation $T_{\iota}$ on ${\dv R}^{\caP}$ (see \S\,\ref{ssec.reflection}) has the same effect
as the application of $T_{\iota}$ on ${\dv R}^{\caPM}$ followed by the application
of the lower replication $\lowrepl$. The roles of $\upprepl$ and $\lowrepl$ can be interchanged here:
\begin{equation}
\forall\,r\in {\dv R}^{\caPM}\qquad
T_{\iota}(\upprepl (r)) =\lowrepl (T_{\iota}(r)) ~~\&~~ T_{\iota}(\lowrepl (r)) =\upprepl (T_{\iota}(r))\,.
\label{eq.dual-repl}
\end{equation}

The motivation for the terminology is similar as in the case of lower replication. In the non-trivial case $|N\setminus L|\geq 2$,
if $r\in {\dv R}^{\caPM}$ is virtually ``decomposed" into
its ``lower" part $T_{-z}(r)$ and its ``upper" part $T_{|z}(r)$ then the function $m=\upprepl (r)$
is obtained by copying once the ``lower" part of $r$ to $\caPL$ and
by (multiple) replication of its ``upper" part to each $R\cup\caPL$ where $\emptyset\neq R\subseteq N\setminus L$.

Without loss of generality, one can limit oneself to the case
$|N\setminus L|=2$ because repeated upper replication is equivalent to
upper replication with a larger set $N\setminus L$; the justification is analogous as in the case of lower replication.
The upper replication can also be interpreted as a simultaneous inversion of
certain ``coinciding" projections. Indeed, if $N\setminus L=\{ x,y\}$
and $m=\upprepl (r)$ then $r$ can be reconstructed from $m$ in three ways:
from the deletion minor $T_{-x}(m)$ by re-naming $y$ to $z$,
from the deletion minor $T_{-y}(m)$ by re-naming $x$ to $z$, and
from the contraction $T_{\{x,y\}\to x}(m)$ by re-naming $x$ to $z$.

\begin{obser}\label{obs.upper-repli}\rm  ~~
Let $N,M$ be sets of variables with $z\in M$, $L:=M\setminus z$ and $N\setminus L= \{ x,y\}$.
Assume $r\in {\dv R}^{\caPM}$.
\begin{itemize}
\item[(i)]
One has $m=\upprepl (r)\in\caK$ iff [\,$r\in\caKM$ and $r(L)\geq r(M)$\,].
The independency model induced by $\upprepl (r)$ is then determined as follows: given $\langle a,b|C\rangle\in\caE$ one has
$a\ci b\,|\,C\,\,[m]$ in the following cases
\begin{eqnarray*}
a\cap L\ci b\cap L\,|\, z\cup (C\cap L)\,\, [r] &\mbox{in case}& \{x,y\}\cap C\neq\emptyset\,,\\
a\ci b\,|\, C\,\, [r] &\mbox{in case}& (a\cup b\cup C)\subseteq L,\\
z\ci \{a,b\}\cap L\,|\, C\,\, [r] &\mbox{in case}& |\{x,y\}\cap \{ a,b\}|=1 ~\&~ \{x,y\}\cap C=\emptyset,\\
r(z\cup C)=r(C) &\mbox{in case}& \{x,y\}=\{ a,b\}\,.
\end{eqnarray*}
\item[(ii)] If $r\in\calK_{u}(M)$ then $\upprepl (r)\in\caKu$.
A function $r\in\calK_{u}(M)$ is extreme in  $\caKM$ iff $\upprepl (r)$ is extreme in $\caK$.
\item[(iii)] If $r$ is extreme in $\caKM$ and $r(L)=r(M)$ then
$\upprepl (r)$ is extreme in $\caK$ and, moreover, $\upprepl (r)\approx\upprepl (r_{u})$. Conversely, if $\upprepl (r)$ is extreme in $\caK$
then either the above case occurs or one has $r(L)>r(M)$, $r\in\caLM$ and
$\upprepl(r)\approx [r(L)-r(M)]\cdot v$, where $v(T):=\delta (T\subseteq L)$ for
$T\subseteq N$, denotes the identifier of subsets of $L$.
\end{itemize}
\end{obser}

\begin{proof}
(i): Denote $r^{\prime}:= T_{\iota}(r)$, apply Observation \ref{obs.reflect}(i) and Observation \ref{obs.lower-repli}(i) which say
\begin{eqnarray*}
\lefteqn{\hspace*{-6mm}m=\upprepl (r)\in\caK ~~\Leftrightarrow ~~
\lowrepl (r^{\prime}) =\lowrepl (T_{\iota}(r))\stackrel{\eqref{eq.dual-repl}}{=}  T_{\iota}(\upprepl (r))\in\caK}\\
&\Leftrightarrow& [\, r^{\prime}\in\caKM ~\&~ r^{\prime}(z)\geq r^{\prime}(\emptyset )\,]
~\Leftrightarrow~ [\, r\in\caKM ~\&~ r(L)\geq r(z\cup L)\,]\,.
\end{eqnarray*}
The form of the induced model also follows from combined Observations \ref{obs.reflect}(i) and \ref{obs.lower-repli}(i).\\[0.4ex]
(ii): realize that $r^{\prime}:= T_{\iota}(r)$ is $\ell$-standardized iff $r$ is $u$-standardized; thus, one can
use the same reasoning procedure and combine Observation \ref{obs.reflect}(i)(ii) with Observation \ref{obs.lower-repli}(ii).\\[0.4ex]
(iii): again, combine Observation \ref{obs.reflect}(ii) with \eqref{eq.dual-repl} and Observation \ref{obs.lower-repli}(iii); one also needs to realize that the reflection $T_{\iota}$ transforms $m^{\{x,y\}\subseteq}$ to the identifier $v$ of subsets of $L$.
\end{proof}

Analogously, one can verify the following fact.

\begin{cor}\label{cor.upper-repli}\rm  ~~
In the situation from Observation \ref{obs.upper-repli},
if $r$ is extreme in $\caKM$ then
$$
\upprepl (r)\approx \upprepl (r_{u})~~\Leftrightarrow~~ \upprepl (r) ~~\mbox{is extreme in $\caK$} ~~\Leftrightarrow~~
r(z\cup L)=r(L)\,.
$$
\end{cor}

\begin{proof}
Combine Observation \ref{obs.reflect}(ii), \eqref{eq.dual-repl} and Corollary \ref{cor.lower-repli}.
\end{proof}

Thus, the upper replication also allows one to produce an extreme supermodular function on $\caP$
on basis of any given extreme supermodular function $r$ on $\caPM$. The point is that one can always take $r_{u}$
and $\upprepl (r_{u})$ is then extreme in $\caK$, by Observation \ref{obs.upper-repli}(ii).
On the top of that, by Corollary \ref{cor.upper-repli}, if $r\approx r^{\prime}\in\caKM$ such that $\upprepl (r^{\prime})$
is extreme in $\caK$ then $\upprepl (r^{\prime})\approx \upprepl (r_{u})$.
Thus, $\upprepl$  establishes a mapping between $\approx$-equivalence classes of extreme supermodular functions.
\medskip

\begin{example}\label{exa.upper-repli}
Consider $N=\{a,b,c,d\}$, $M=\{ b,c,d\}$ and $z=b$, that is, $L=\{ c,d\}$. Take
$r=\delta_{\{ b,c,d\}}+\delta_{\{ b,c\}}$, which is an extreme supermodular function on $\caPM$. Then $r_{u}=\delta_{d}+\delta_{\emptyset}\in {\dv R}^{\caPM}$.
The application of \eqref{eq.upp-replica-def} gives
$m= \upprepl (r_{u})=\delta_{d}+\delta_{\emptyset}\in {\dv R}^{\caP}$, which is
by Observation \ref{obs.upper-repli}(ii), an extreme $u$-standardized supermodular function on $\caP$.
Its $\ell$-standardized version is then
$$
m_{u} =2\cdot\delta_{\{ a,b,c,d\}}+ 2\cdot\delta_{\{ a,b,c\}}+ \delta_{\{ a,b,d\}}+\delta_{\{ a,c,d\}}+\delta_{\{ b,c,d\}}
+ \delta_{\{ a,b\}} + \delta_{\{ a,c\}} + \delta_{\{ b,c\}}\,,
$$
which is an extreme $\ell$-standardized supermodular function.
\end{example}



\appendix

\section{Facts on faces}\label{sec.face}

\begin{defin}\rm
Consider the linear space ${\dv R}^{\caP}$; let $K\subseteq {\dv R}^{\caP}$ be a convex cone.\\
The cone $K$ is called {\em pointed} if $K\cap (-K)=\{0\}$.\\
A {\em valid\/} inequality for $K$ is an inequality
(for vectors $m\in {\dv R}^{\caP}$) of the form
$$
k\leq \langle m,v\rangle\quad
\mbox{where $v\in {\dv R}^{\caP}$ and $k\in {\dv R}$ are such that $K\subseteq \{ m\in  {\dv R}^{\caP} :~ k\leq \langle m,v\rangle\,\}$}.
$$
A {\em face\/} of $K$ is its subset $F\subseteq K$ of the form
$$
F= \{ m\in K :~ k=\langle m,v\rangle\,\}\quad
\mbox{where $v\in {\dv R}^{\caP}$ and $k\in {\dv R}$ define a valid inequality for $K$}.
$$
The {\em dimension\/} of a face $F$ is the dimension of its affine hull.\\
The {\em relative interior\/} of a face $F$ is
$$
\inte (F) := F\setminus \bigcup\, \{ L :~ \mbox{$L$ is a face of $K$ such that $L\subset F$}\,\}\,.
$$
\end{defin}

\begin{obser}\rm
Every non-empty face $F$ of a convex cone $K$ is defined by an inequality of the form
$0\leq \langle m,v\rangle$. If an inner point of a segment in $K$ belongs a face $F$ (of $K$)
then the whole segment belongs to $F$.
\end{obser}

\begin{proof}
Indeed, we know that $m^{0}:=0\in K$ and, thus, $k\leq\langle m^{0},v\rangle =0$. Assume for a contradiction
$k<0$; the non-emptiness of $F$ then implies the existence of $y\in F\subseteq K$ with
$k=\langle y,v\rangle$. Since $K$ is a cone, $x=2\cdot y\in K$ and
$k> 2\cdot k = \langle x,v\rangle$ for which reason $k\leq \langle m,v\rangle$ does not hold
for $m=x\in K$, which is a contradiction.

To show the second statement, assume $m^{1},m^{2}\in K$ such that
$\alpha\cdot m^{1} + (1-\alpha )\cdot m^{2}\in F$ for some $\alpha\in (0,1)$. That means, by linearity of
scalar product, $0=\alpha\cdot \langle m^{1},v\rangle +(1-\alpha )\cdot \langle m^{2},v\rangle$. Since the
respective inequality
is valid for $K$ we know $0\leq\langle m^{1},v\rangle$ and $0\leq\langle m^{2},v\rangle$, which
together with the equality forces $\langle m^{1},v\rangle =0=\langle m^{2},v\rangle$, that is, $m^{1},m^{2}\in F$.
\end{proof}

{\bf Basic facts on the face lattice}\\[0.4ex]
The class of non-empty faces of a {\em polyhedral cone} $K\subseteq {\dv R}^{\caP}$,
that is, of a cone $K$ defined by means of a finite number of inequalities
of the type $0\leq\langle m,v^{i}\rangle$ with $v^{i}\in {\dv R}^{\caP}$, forms a finite lattice relative
to the inclusion order $\subseteq$. This {\em face lattice\/} has the following properties.
\begin{itemize}
\item It is a {\em graded lattice} in which the dimension of a face is its grade in the lattice.
\item The {\em one} of the lattice, that is, the largest element (= face) is the whole cone $K$, with
the maximal dimension $\dim (K)$.
\item The {\em zero\/} of the lattice, that is, the smallest element (= face) is the linear
subspace $L=K\cap (-K)$, with the minimal dimension $\dim (L)$. In case of a pointed cone $K$ the dimension of $L$ is $0$.
\end{itemize}
In case $K\neq L$ we, moreover, distinguish atoms and co-atoms of the face lattice.
\begin{itemize}
\item The {\em atoms\/} of the lattice are the upper neighbors of the
zero in the lattice. This is such a face $F$ that the only proper sub-face $G\subset F$ is the linear space $G=L\equiv K\cap (-K)$. Since the face lattice is graded, an equivalent
characterization is that it is a face of the dimension $\dim (L)+1$. In case of a pointed cone $K$ it a face of the dimension $1$, that is,
an {\em extreme ray\/} of $K$. An equivalent definition of an extreme ray is as follows. A ray $R=\{ \lambda\cdot m\, :\ \lambda\geq 0\}$, $m\in K$ is
extreme if every segment in $K$ whose inner point belongs to $R$ must whole belong to $R$; formally: for every $m^{1},m^{2}\in K$,
if [$\alpha\cdot m^{1}+ (1-\alpha )\cdot m^{2}=m$ for $\alpha\in (0,1)$] then
[$m^{i}=\lambda^{i}\cdot m$ with $\lambda^{i}\geq 0,~ i=1,2$].
\item The {\em co-atoms\/} of the lattice are the lower neighbors of
the one in the lattice. This is such a face
$F$ that the only proper super-face $G\supset F$ is the whole cone $G=K$. Since the face lattice is graded, an equivalent characterization is that it is a face of the dimension $\dim (K)-1$. Such faces are called {\em facets\/}.
Since the face lattice is co-atomic, that is, every element is an infimum of
(a set of) co-atoms, we know that every face $F$ is the intersection of facets (containing $F$). Of course, in case of the largest face $F=K$ (= the one of the lattice) it is the intersection of the empty system of facets.
\end{itemize}
\smallskip

In case $K$ is a polyhedral cone, that is, defined by a list of linear inequalities
$$
K = \{ m\in {\dv R}^{\caP} :~ 0\leq \langle m,v^{i}\rangle\,\}\quad
\mbox{for $v^{1},\ldots, v^{t}\in  {\dv R}^{\caP}\setminus\{0\}$, $t\geq 1$,}
$$
one knows that the facets of $K$ are determined by those inequalities. More precisely, facets of $K$ are given by just those inequalities $0\leq \langle m,v^{i}\rangle$ that are not redundant, that is, not derivable from the others
(by a conic combination of inequalities).
\end{document}